\documentclass[11pt,a4paper]{amsart}

\usepackage{amsmath,amssymb,a4wide,amsthm,enumerate,enumitem,hyperref,ifthen,xcolor,dsfont}
\usepackage[utf8]{inputenc}
\usepackage{graphicx}

\newtheorem{theo}{Theorem}
\newtheorem{prop}[theo]{Proposition}

\newtheorem{lemma}[theo]{Lemma}
\theoremstyle{remark}
\newtheorem{remark}{Remark}

\usepackage{natbib}

\def\Var{\mathop{\rm Var}\nolimits}
\def\Cov{\mathop{\rm Cov}\nolimits}
\def\rset{\mathbb{R}}
\def\PE{\mathbb{E}}
\def\PP{\mathbb{P}}
\def\X{\mathcal{X}}
\def\E{\mathcal{E}}
\def\B{\mathcal{B}}

\def\lambdaMax{\lambda_{\rm{max}}}
\def\lambdaMin{\lambda_{\rm{min}}}

\newcommand{\sign}{\ensuremath{\rm sign}}

\begin{document}

\title[Variable selection in multivariate linear models]{Variable selection in multivariate linear models with high-dimensional covariance
matrix estimation}

\date{\today}

\author{M.\,Perrot-Dockès}
\address{UMR MIA-Paris, AgroParisTech, INRA, Universit\'e Paris-Saclay, 75005, Paris, France}
\email{marie.perrot-dockes@agroparistech.fr}
\author{C.\,L\'evy-Leduc}
\address{UMR MIA-Paris, AgroParisTech, INRA, Universit\'e Paris-Saclay, 75005, Paris, France}
\email{celine.levy-leduc@agroparistech.fr}
\author{L.\,Sansonnet}
\address{UMR MIA-Paris, AgroParisTech, INRA, Universit\'e 
Paris-Saclay, 75005, Paris, France}
\email{laure.sansonnet@agroparistech.fr}
\author{J.\,Chiquet}
\address{UMR MIA-Paris, AgroParisTech, INRA, Universit\'e 
Paris-Saclay, 75005, Paris, France}
\email{julien.chiquet@agroparistech.fr}

\begin{abstract}
In this paper, we propose a novel variable selection approach in the framework of multivariate linear models
taking into account the dependence that may exist between the responses.
It consists in estimating beforehand the covariance matrix $\Sigma$ of the responses and 
to plug this estimator in a Lasso criterion, in order to obtain a sparse estimator of
the coefficient matrix. The properties of our approach are investigated both from a theoretical and
a numerical point of view. More precisely, we give general conditions that the estimators of
the covariance matrix and its inverse have to satisfy in order to recover the positions of the
null and non null entries of the coefficient matrix when the size of $\Sigma$ is not fixed and can tend to infinity.
We prove that these conditions are satisfied in the particular case of some Toeplitz matrices.
Our approach is implemented in the R package \textsf{MultiVarSel} available from the Comprehensive R Archive Network (CRAN) and
 is very attractive since it benefits from a low computational load. 
We also assess the performance of our methodology using synthetic data and compare it with alternative approaches.
Our numerical experiments show that including the estimation of the covariance matrix
in the Lasso criterion dramatically improves the variable selection performance in many cases.
\end{abstract}

\maketitle

\section{Introduction}

The multivariate linear model consists in generalizing the classical linear model, in which a single
response is explained by $p$ variables, to the case where the number $q$ of responses is larger than 1.
Such a general modeling can be used in a wide variety of applications ranging from econometrics (\cite{lutkepohl:2005}) to
bioinformatics (\cite{Meng:2014}). In the latter field, for instance, multivariate models have been used to gain insight into complex
biological mechanisms like metabolism or gene regulation. This has been made possible thanks to recently developed sequencing technologies. 
For further details, we refer the reader to \cite{mehmood:liland:2012}.
However,  the downside of such a technological expansion is to include irrelevant variables in the statistical models. 
To circumvent this, devising efficient variable selection approaches in the multivariate setting has become 
a growing concern.


A first naive approach to deal with the variable selection issue in the multivariate setting consists in applying classical univariate variable selection strategies to each response 
separately. Some well-known variable selection methods include the least absolute shrinkage
and selection operator (LASSO) proposed by \cite{Tib96} and the smoothly clipped absolute deviation (SCAD) approach devised by
\cite{fan:li:2001}. However, such a strategy does not take into account the dependence that may exist between the different responses.

In this paper, we shall consider the following multivariate linear model:
\begin{equation} \label{eq:model:matriciel}
    Y=XB+E,
\end{equation}
where $Y=(Y_{i,j})_{1 \leq i \leq n,1 \leq j \leq q}$ denotes the $n\times q$ random response matrix, $X$ denotes the $n \times p$ design matrix,  $B$ denotes a $p \times q$ coefficient matrix
and $E=(E_{i,j})_{1 \leq i \leq n,1 \leq j \leq q}$ denotes the $n\times q$ random error matrix, where $n$ is the sample size. In order to model the potential dependence that may exist between the columns of $E$, we shall assume that for each $i$ in $\{1,\ldots,n\}$,
\begin{equation} \label{eq:def_E}
    (E_{i,1},\ldots,E_{i,q}) \sim \mathcal{N}(0,\Sigma),
\end{equation}
where $\Sigma$ denotes the covariance matrix of the $i$th row of the error matrix $E$. 
We shall moreover assume that the different rows of $E$ are independent. 
With such assumptions, there is some dependence between the columns of $E$ but not between the rows.
Our goal is here to design a variable selection approach which is able to identify the positions of the null and non null entries
in the sparse matrix $B$ by taking into account the dependence between the columns of $E$.

This issue has recently been considered by \cite{Lee-Liu-2012} who extended the approach of \cite{rothman:2010}.
More precisely, \cite{Lee-Liu-2012} proposed three approaches for dealing with this issue
based on penalized maximum likelihood with a weighted $\ell_1$ regularization. In their first approach
$B$ is estimated by using a plug-in estimator of $\Sigma^{-1}$, in the second one,  $\Sigma^{-1}$ is estimated 
by using a plug-in estimator of $B$ and in the third one, $\Sigma^{-1}$ and $B$ are estimated simultaneously.
\cite{Lee-Liu-2012} also investigate the asymptotic properties of their methods when the sample size $n$ tends to infinity
and the number of rows and columns $q$ of $\Sigma$ is fixed.

In this paper, we propose to estimate $\Sigma$ beforehand and to plug this estimator in a Lasso criterion, in order to obtain a sparse estimator of $B$.
Hence, our methodology is close to the first approach of \cite{Lee-Liu-2012}. However, there are two main differences:
The first one is the asymptotic framework in which our theoretical results are established and the second one is the strategy that we use for estimating $\Sigma$.
More precisely, in our asymptotic framework, $q$ is allowed to depend on $n$ and thus to tend to infinity as $n$ tends to infinity at a polynomial rate.
Moreover, in \cite{Lee-Liu-2012}, $\Sigma^{-1}$ is estimated by using an adaptation of the Graphical Lasso (GLASSO) proposed by \cite{friedman:hastie:tibshirani:2008}. This technique
has also been considered by \cite{yuan:li:2007}, \cite{banerjee:2008} and \cite{rothman:2008}. In this paper, we give general conditions that the estimators of $\Sigma$ and $\Sigma^{-1}$ have 
to satisfy in order to be able to recover the support of $B$
that is to find the positions of the null and non null entries of the matrix $B$. We prove that when $\Sigma$ is a particular Toeplitz matrix, namely the covariance matrix of
an AR(1) process, the assumptions of the theorem are satisfied.

Let us now describe more precisely our methodology. We start by ``whitening'' the observations $Y$ by applying the following transformation to Model (\ref{eq:model:matriciel}):
\begin{equation} \label{eq:modele:blanchi_est}
	Y\,\Sigma^{-1/2} = XB\,\Sigma^{-1/2} + E\,\Sigma^{-1/2}.
\end{equation}
The goal of such a transformation is to remove the dependence between the columns of $Y$. Then, for estimating $B$, we proceed as follows. Let us observe that
(\ref{eq:modele:blanchi_est}) can be rewritten as:
\begin{equation} \label{eq:model_vec}
	\mathcal{Y}=\mathcal{X}\mathcal{B}+\mathcal{E},
\end{equation}
with
\begin{equation} \label{eq:def:Y_X_E}
    \mathcal{Y}=vec(Y\,\Sigma^{-1/2}), \,
    \mathcal{X}=(\Sigma^{-1/2})' \otimes X, \,
    \mathcal{B}=vec(B) \textrm{\, and \,}
    \mathcal{E}=vec(E\,\Sigma^{-1/2}),
\end{equation}
where $vec$ denotes the vectorization operator and $\otimes$ the Kronecker product.

With Model~\eqref{eq:model_vec}, estimating $B$ is equivalent to estimate $\mathcal{B}$ since $\mathcal{B}=vec(B)$. Then, for estimating $\mathcal{B}$, we use the classical LASSO criterion defined as follows for a nonnegative $\lambda$:
\begin{equation} \label{eq:lasso}
    \widehat{\mathcal{B}}(\lambda) = \textrm{Argmin}_\mathcal{B}
    \left\{\|\mathcal{Y}-\mathcal{X}\mathcal{B}\|_2^2 +
    \lambda\|\mathcal{B}\|_1\right\},
\end{equation}
where $\|.\|_1$ and $\|.\|_2$ denote the classical $\ell_1$-norm and $\ell_2$-norm, respectively. Inspired by \cite{zhao:yu:2006}, Theorem~\ref{th1} established some conditions under which the positions of the null and non null entries of $\mathcal{B}$ can be recovered by using $\widehat{\mathcal{B}}$.

In practical situations, the covariance matrix $\Sigma$ is generally unknown and has thus to be estimated.  Let $\widehat{\Sigma}$ denote 
an estimator of $\Sigma$. Then, the estimator $\widehat{\Sigma}^{-1/2}$ of $\Sigma^{-1/2}$ is such that
$$
\widehat{\Sigma}^{-1}=\widehat{\Sigma}^{-1/2} (\widehat{\Sigma}^{-1/2})'.
$$
When $\Sigma^{-1/2}$ is replaced by $\widehat{\Sigma}^{-1/2}$, \eqref{eq:modele:blanchi_est} becomes
\begin{equation} \label{eq:modele:blanchi_est_hat}
    Y\,\widehat{\Sigma}^{-1/2} = XB\,\widehat{\Sigma}^{-1/2} + E\,\widehat{\Sigma}^{-1/2},
\end{equation}	
which can be rewritten as follows:	
\begin{equation} \label{eq:model_vec_tilde}
    \widetilde{\mathcal{Y}} = \widetilde{\mathcal{X}}\mathcal{B} + \widetilde{\mathcal{E}},
\end{equation}
where
\begin{equation} \label{eq:def:Y_X_E_tilde}
    \widetilde{\mathcal{Y}} = vec(Y\,\widehat{\Sigma}^{-1/2}), \,
    \widetilde{\mathcal{X}} = (\widehat{\Sigma}^{-1/2})' \otimes X, \,
    \mathcal{B} = vec(B) \textrm{\, and \,}
    \widetilde{\mathcal{E}} = vec(E\,\widehat{\Sigma}^{-1/2}).
\end{equation}
In Model~\eqref{eq:model_vec_tilde}, $\B$ is estimated by
\begin{equation} \label{eq:lasso_tilde}
    \widetilde{\mathcal{B}}(\lambda) = \textrm{Argmin}_\mathcal{B}
    \left\{\|\widetilde{\mathcal{Y}}-\widetilde{\mathcal{X}}\mathcal{B}\|_2^2 +
    \lambda\|\mathcal{B}\|_1\right\}.
\end{equation}
By extending Theorem~\ref{th1}, Theorem~\ref{th2} gives some conditions on the eigenvalues of $\Sigma^{-1}$ and on the convergence rate of $\widehat{\Sigma}$ and its inverse to $\Sigma$ and
$\Sigma^{-1}$, respectively, under which the positions of the null and non null entries of $\mathcal{B}$ can be recovered by using $\widetilde{\mathcal{B}}$.
 
We prove in Section \ref{subsec:AR1} that when $\Sigma$ is a particular Toeplitz matrix, namely the covariance matrix of
an AR(1) process, the assumptions of Theorem~\ref{th2} are satisfied. This strategy has been implemented in the R package \textsf{MultiVarSel}, which is available on the Comprehensive R Archive
Network (CRAN), for more general Toeplitz matrices $\Sigma$ such as the covariance matrix of ARMA processes or general stationary processes.
For a successful application of this methodology to particular ``-omic'' data, namely metabolomic data, we refer the reader to  \cite{nous:bioinfo:arxiv:2017}.
For a review of the most recent methods for estimating high-dimensional covariance matrices,
we refer the reader to \cite{HDCovEst}.

The paper is organized as follows.
Section~\ref{sec:theo} is devoted to the theoretical results of the paper.
The assumptions under which the positions of the non null and null entries of $\mathcal{B}$ can be recovered are established in 
Theorem~\ref{th1} when $\Sigma$ is known and in Theorem~\ref{th2} when $\Sigma$ is unknown.
Section~\ref{subsec:AR1} studies the specific case of the AR(1) model.
We present in Section \ref{sec:num_exp} some numerical experiments in order to support our theoretical results. 
The proofs of our main theoretical results are given in Section~\ref{sec:proofs}.

\section{Theoretical results}\label{sec:theo}

\subsection{Case where $\Sigma$ is known}

Let us first introduce some notations. Let

\begin{equation} \label{eq:C:J}
    C=\frac{1}{nq}\mathcal{X}'\mathcal{X} \textrm{\, and \,}
    J=\{1 \leq j \leq pq, \mathcal{B}_j \neq 0\}, 
\end{equation}
where $\mathcal{X}$ is defined in (\ref{eq:def:Y_X_E})  and where $\mathcal{B}_j$ denotes the $j$th component of the vector $\mathcal{B}$ defined in \eqref{eq:def:Y_X_E}. 

Let also define
\begin{equation} \label{eq:CJJ}
    C_{J,J} = \frac{1}{nq} (\mathcal{X}_{\bullet,J})'\mathcal{X}_{\bullet,J} \textrm{\, and \,} C_{J^c,J} = \frac{1}{nq} (\mathcal{X}_{\bullet,J^c})'\mathcal{X}_{\bullet,J},
\end{equation}
where $\mathcal{X}_{\bullet,J}$ and $\mathcal{X}_{\bullet,J^c}$ denote the columns of $\mathcal{X}$ belonging to the set $J$ defined in (\ref{eq:C:J}) and to its complement $J^c$, respectively.

More generally, for any matrix $A$, $A_{I,J}$ denotes the partitioned matrix extracted from $A$ by considering the rows of $A$ belonging to the set $I$ and 
the columns of $A$ belonging to the set $J$, with $\bullet$ indicating all the rows or all the columns.

The following theorem gives some conditions under which the estimator $\widehat{\mathcal{B}}$ defined in (\ref{eq:lasso}) is sign-consistent as defined by \cite{zhao:yu:2006}, namely,
$$
\PP\left(\sign(\widehat{\mathcal{B}}) = \sign(\mathcal{B})\right)\to 1, \textrm{\, as \,} n\to\infty,
$$
where the $\sign$ function maps positive entries to 1, negative entries to -1 and zero to 0.

\begin{theo}\label{th1}
Assume that $\mathcal{Y}=(\mathcal{Y}_1,\mathcal{Y}_2,\ldots,\mathcal{Y}_{nq})'$ satisfies Model \eqref{eq:model_vec}. Assume also that
there exist some positive constants $M_1$, $M_2$, $M_3$ and positive numbers $c_1$, $c_2$ such that $0<c_1 +c_2<1/2$ satisfying:
\begin{enumerate}[label=\textbf{\rm (A\arabic*)}]
\item \label{th1(i)}  for all $n\geq 1$, for all $j\in\{1,\ldots,pq\}$, $\frac{1}{n} (\mathcal{X}_{\bullet,j})'\mathcal{X}_{\bullet,j} \leq  M_1$, where $\mathcal{X}_{\bullet,j}$ is the $j$th column of $\mathcal{X}$ defined in (\ref{eq:def:Y_X_E}),
\item \label{th1(ii)}  for all $n\geq 1$, $\frac{1}{n}\lambdaMin\left((\mathcal{X}'\mathcal{X})_{J,J}\right)\geq M_2$, where $\lambdaMin(A)$ denotes the smallest eigenvalue of $A$,
\item \label{th1(iii)} $|J|=O(q^{c_1})$, where $J$ is defined in (\ref{eq:C:J}) and $|J|$ is the cardinality of the set $J$,
\item \label{th1(iv)} $q^{c_2} \min_{j \in J} |\mathcal{B}_j| \geq M_3$.
\end{enumerate}
Assume also that the following strong Irrepresentable Condition holds:
\begin{enumerate}[label=\textbf{\rm (IC)}]
\item \label{eq:irrep} There exists a positive constant vector $\eta$ such that
$$
\left|(\mathcal{X}'\mathcal{X})_{J^c,J}((\mathcal{X}'\mathcal{X})_{J,J})^{-1} \, \sign(\mathcal{B}_J)\right|\leq \mathbf{1}-\eta,
$$
where $\mathbf{1}$ is a $(pq-|J|)$ vector of 1 and the inequality holds element-wise.
\end{enumerate}
Then, for all $\lambda$ that satisfies
\begin{enumerate}[label=\textbf{\rm (L)}]
\item \label{assum:lambda_nq}
$\displaystyle \quad q=q_n=o\left(n^{\frac{1}{2(c_1+c_2)}}\right), \quad \frac{\lambda}{\sqrt{n}}\to\infty \textrm{\, and \,} \frac{\lambda}{n} =o\left(q^{-(c_1+c_2)}\right),\textrm{ as } n\to\infty,$
\end{enumerate}
we have
$$
\PP\left(\sign(\widehat{\mathcal{B}}(\lambda)) = \sign(\mathcal{B})\right)\to 1, \textrm{\, as \,} n\to\infty,
$$
where $\widehat{\mathcal{B}}(\lambda)$ is defined by (\ref{eq:lasso}).
\end{theo}

\begin{remark}
Observe that if $c_1+c_2<(2k)^{-1}$, for some positive $k$, then the first condition of \ref{assum:lambda_nq} becomes $q=o(n^k)$. Hence for large values of $k$, the size $q$ of $\Sigma$ 
is much larger than $n$.
\end{remark}

The proof of Theorem~\ref{th1} is given in Section \ref{sec:proofs}.
It is based on Proposition~\ref{prop1} which is an adaptation to the multivariate case of Proposition 1 in \cite{zhao:yu:2006}.

\begin{prop}\label{prop1}
Let $\widehat{\mathcal{B}}(\lambda)$ be defined by (\ref{eq:lasso}). Then
$$
\PP\left(\sign(\widehat{\mathcal{B}}(\lambda))=\sign(\mathcal{B})\right)\geq\PP(A_n\cap B_n),
$$
where
\begin{equation}\label{eq:An}
A_n=\left\{\left|(C_{J,J})^{-1}W_J\right|<\sqrt{nq}\left(|\mathcal{B}_J|-\frac{\lambda}{2nq}|(C_{J,J})^{-1}\sign(\mathcal{B}_J)|\right)\right\}
\end{equation}
and
\begin{equation}\label{eq:Bn}
B_n=\left\{\left|C_{J^c,J}(C_{J,J})^{-1}W_J-W_{J^c}\right|\leq \frac{\lambda}{2\sqrt{nq}}\left(\mathbf{1}-\left|C_{J^c,J}(C_{J,J})^{-1}\sign(\mathcal{B}_J)\right|\right)
\right\},
\end{equation}
with $
W=\mathcal{X}'\mathcal{E}/\sqrt{nq}. 
$
In (\ref{eq:An}) and (\ref{eq:Bn}), $C_{J,J}$ and $C_{J^c,J}$ are defined in (\ref{eq:CJJ}) and $W_J$ and $W_{J^c}$ denote the components of $W$ being in $J$ and $J^c$, respectively.
Note that the previous inequalities hold element-wise.
\end{prop}

The proof of Proposition \ref{prop1} is given in Section \ref{sec:proofs}.

We give in the following proposition which is proved in Section \ref{sec:proofs} some conditions on $X$ and $\Sigma$ under which Assumptions~\ref{th1(i)}~and~\ref{th1(ii)} of Theorem~\ref{th1} hold.

\begin{prop}\label{propcond}
If there exist some positive constants $M'_1$, $M'_2$, $m_1$, $m_2$ such that,  for all $n\geq 1$,
\begin{enumerate}[label=\textbf{\rm (C\arabic*)}]
\item \label{cond1} for all $j\in\{1,\ldots,p\}$, $\frac{1}{n} (X_{\bullet,j})'X_{\bullet,j}\leq M'_1$,
\item \label{cond2} $\frac{1}{n}\lambdaMin(X'X) \geq M'_2$,
\item \label{cond3} $\lambdaMax(\Sigma^{-1}) \leq m_1$,
\item \label{cond4} $\lambdaMin(\Sigma^{-1}) \geq m_2$, 
\end{enumerate}
then Assumptions~\ref{th1(i)}~and~\ref{th1(ii)} of Theorem~\ref{th1} are satisfied. 
\end{prop}

\begin{remark}
Observe that \ref{cond1} and \ref{cond2} hold in the case where the columns of the matrix $X$ are orthogonal.
\end{remark}

We give in Proposition \ref{prop_IC} in Section \ref{subsec:AR1} some conditions under which Condition \ref{eq:irrep} holds in the specific case where $\Sigma$ is the covariance
matrix of an AR(1) process.

\subsection{Case where $\Sigma$ is unknown}

Similarly as in (\ref{eq:C:J}) and (\ref{eq:CJJ}), we introduce the following notations:
\begin{equation}\label{eq:C_tilde}
\widetilde{C}=\frac{1}{nq}\widetilde{\mathcal{X}}'\widetilde{\mathcal{X}}
\end{equation}
and
\begin{equation} \label{eq:CJJ_tilde}
    \widetilde{C}_{J,J} = \frac{1}{nq} (\widetilde{\mathcal{X}}_{\bullet,J})'\widetilde{\mathcal{X}}_{\bullet,J} \textrm{\, and \,} \widetilde{C}_{J^c,J} = \frac{1}{nq} (\widetilde{\mathcal{X}}_{\bullet,J^c})'\widetilde{\mathcal{X}}_{\bullet,J},
\end{equation}
where $\widetilde{\mathcal{X}}_{\bullet,J}$ and $\widetilde{\mathcal{X}}_{\bullet,J^c}$ denote the columns of $\widetilde{\mathcal{X}}$ belonging to the set $J$ defined in \eqref{eq:C:J} and to its complement $J^c$, respectively.

A straightforward extension of Proposition \ref{prop1} leads to the following proposition for Model~(\ref{eq:model_vec_tilde}).

\begin{prop} \label{prop2}
Let $\widetilde{\mathcal{B}}(\lambda)$ be defined by (\ref{eq:lasso_tilde}). Then
$$
\PP\left(\sign(\widetilde{\mathcal{B}}(\lambda))=\sign(\mathcal{B})\right)\geq\PP(\widetilde{A_n}\cap \widetilde{B_n}),
$$
where
\begin{equation}\label{eq:An_tilde}
\widetilde{A_n}=\left\{\left|(\widetilde{C}_{J,J})^{-1}\widetilde{W}_J\right|<\sqrt{nq}\left(|\mathcal{B}_J|-\frac{\lambda}{2nq}|(\widetilde{C}_{J,J})^{-1}\sign(\mathcal{B}_J)|\right)\right\}
\end{equation}
and
\begin{equation}\label{eq:Bn_tilde}
\widetilde{B_n}=\left\{\left|\widetilde{C}_{J^c,J}(\widetilde{C}_{J,J})^{-1}\widetilde{W}_J-\widetilde{W}_{J^c}\right|\leq \frac{\lambda}{2\sqrt{nq}}\left(\mathbf{1}-\left|\widetilde{C}_{J^c,J}(\widetilde{C}_{J,J})^{-1}\sign(\mathcal{B}_J)\right|\right)
\right\},
\end{equation}
with $
W=\widetilde{\mathcal{X}}'\widetilde{\mathcal{E}}/\sqrt{nq}. 
$
In (\ref{eq:An_tilde}) and (\ref{eq:Bn_tilde}), $\widetilde{C}_{J,J}$ and $\widetilde{C}_{J^c,J}$ are defined in (\ref{eq:CJJ_tilde}) and $\widetilde{W}_J$ and $\widetilde{W}_{J^c}$ denote the components of $\widetilde{W}$ being in $J$ and $J^c$, respectively.
Note that the previous inequalities hold element-wise.
\end{prop}

The following theorem extends Theorem \ref{th1} to the case where $\Sigma$ is unknown and gives some conditions under which the estimator $\widetilde{\mathcal{B}}$ defined in
(\ref{eq:lasso_tilde}) is sign-consistent. The proof of Theorem \ref{th2} 
is given in Section \ref{sec:proofs} and is based on Proposition \ref{prop2}.

\begin{theo}\label{th2}
Assume that Assumptions \ref{th1(i)}, \ref{th1(ii)}, \ref{th1(iii)}, \ref{th1(iv)}, \ref{eq:irrep} and \ref{assum:lambda_nq} of Theorem~\ref{th1} hold. Assume also that, there exist some positive
constants $M_4$, $M_5$, $M_6$ and $M_7$, such that for all $n\geq 1$,
\begin{enumerate}[label=\textbf{\rm (A\arabic*)}] \setcounter{enumi}{4}
\item \label{th2(v)}  $\|(X'X)/n\|_\infty\leq M_4$, 
\item \label{th2(vi)}  $\lambdaMin((X'X)/n)\geq M_5$, 
\item \label{th2(vii)} $\lambdaMax(\Sigma^{-1}) \leq M_6$, 
\item \label{th2(viii)} $\lambdaMin(\Sigma^{-1}) \geq M_7$.
\end{enumerate}
Suppose also that
\begin{enumerate}[label=\textbf{\rm (A\arabic*)}] \setcounter{enumi}{8}
\item \label{th2(ix)} $\|\Sigma^{-1}-\widehat{\Sigma}^{-1}\|_{\infty}=O_P((nq)^{-1/2})$, as $n$ tends to infinity,
\item \label{th2(x)} $\rho(\Sigma-\widehat{\Sigma})=O_P((nq)^{-1/2})$, as $n$ tends to infinity.
\end{enumerate}
Let $\widetilde{\mathcal{B}}(\lambda)$ be defined by (\ref{eq:lasso_tilde}), then
$$
\PP\left(\textrm{sign}(\widetilde{\mathcal{B}}(\lambda))=\textrm{sign}(\mathcal{B})\right)\to 1,\textrm{ as } n\to\infty.
$$
In the previous assumptions, $\lambdaMax(A)$, $\lambdaMin(A)$, $\rho(A)$ and $\|A\|_\infty$ denote the largest eigenvalue, the smallest
eigenvalue, the spectral radius  and the infinite norm (induced by the associated vector norm) of the matrix $A$.
\end{theo}

\begin{remark}
Observe that Assumptions \ref{th2(v)} and  \ref{th2(vi)} hold in the case where the columns of the matrix $X$ are orthogonal.
Note also that \ref{th2(vii)} and \ref{th2(viii)} are the same as \ref{cond3} and \ref{cond4} in Proposition \ref{propcond}.
\end{remark}

In order to estimate $\Sigma$, we propose the following strategy: 
\begin{itemize}
\item Fitting a classical linear model to each column of the matrix 
  $Y$ in order to have access to an
  estimation $\widehat{E}$ of the random error matrix
  $E$. It is possible since $p$ is assumed to be fixed and smaller than $n$. 
\item Estimating $\Sigma$ from $\widehat{E}$ by assuming that $\Sigma$ has a particular structure, Toeplitz for instance.
\end{itemize}

More precisely, $\widehat{E}$ defined in the first step is such that:
\begin{equation}\label{eq: Ehat}
\widehat{E}=\left(\textrm{Id}_{\rset^n}-X(X'X)^{-1}X'\right) E=:\Pi E,
\end{equation}
which implies that
\begin{equation}\label{eq:Ehat_1}
\widehat{\mathcal{E}}=vec(\widehat{E})=\left[\textrm{Id}_{\rset^q}\otimes\Pi\right]\mathcal{E},
\end{equation}
where $\mathcal{E}$ is defined in (\ref{eq:def:Y_X_E}).

We prove in Proposition \ref{prop:cond_th_AR} below that our strategy for estimating $\Sigma$ provides an estimator satisfying the assumptions of Theorem \ref{th2}
in the case where $(E_{1,t})_t$, $(E_{2,t})_t$, ..., $(E_{n,t})_t$ are assumed to be independent AR(1) processes.

\subsection{The AR(1) case}\label{subsec:AR1}

\subsubsection{Sufficient conditions for Assumption \ref{eq:irrep} of Theorem \ref{th1}}

The following proposition gives some conditions under which the strong Irrepresentable Condition \ref{eq:irrep} of Theorem \ref{th1} holds.

\begin{prop}\label{prop_IC}
Assume that $(E_{1,t})_t$, $(E_{2,t})_t$, ..., $(E_{n,t})_t$ in Model (\ref{eq:model:matriciel}) are independent AR(1) processes satisfying:
$$
E_{i,t}-\phi_1 E_{i,t-1}=Z_{i,t},\; \forall i\in\{1,\dots,n\},
$$
where the $Z_{i,t}$'s are zero-mean i.i.d. Gaussian random variables with variance $\sigma^2$ and $|\phi_1|<1$.
Assume also that $X$ defined in (\ref{eq:model:matriciel}) is such that $X'X=\nu\textrm{Id}_{\rset^p}$, where $\nu$ is a positive constant. 
Moreover, suppose that if $j\in J$, then $j>p$ and $j< pq -p$. Suppose also that for all $j$, $j-p$ or $j+p$ is not in $J$.
Then, the strong Irrepresentable Condition \ref{eq:irrep} of Theorem \ref{th1} holds.
\end{prop}

The proof of Proposition \ref{prop_IC} is given in Section \ref{sec:proofs}.

\subsubsection{Sufficient conditions for Assumptions \ref{th2(vii)}, \ref{th2(viii)}, \ref{th2(ix)} and \ref{th2(x)} of Theorem \ref{th2}}

The following proposition establishes that in the particular case where the  $(E_{1,t})_t$, $(E_{2,t})_t$, ..., $(E_{n,t})_t$ are independent AR(1) processes, our strategy for estimating $\Sigma$ 
provides an estimator satisfying the assumptions of Theorem \ref{th2}.

\begin{prop}\label{prop:cond_th_AR}
Assume that $(E_{1,t})_t$, $(E_{2,t})_t$, ..., $(E_{n,t})_t$ in Model (\ref{eq:model:matriciel}) are independent AR(1) processes satisfying:
$$
E_{i,t}-\phi_1 E_{i,t-1}=Z_{i,t},\; \forall i\in\{1,\dots,n\},
$$
where the $Z_{i,t}$'s are zero-mean i.i.d. Gaussian random variables with variance $\sigma^2$ and $|\phi_1|<1$.
Let 
$$
\widehat{\Sigma}=\frac{1}{1-\widehat{\phi}_1^2}\left(
\begin{array}{cccccc}
1                & \widehat{\phi}_1 & \widehat{\phi}_1^2  & \dots & \widehat{\phi}_1^{q-1}\\
\widehat{\phi}_1 &  1               & \widehat{\phi}_1   & \dots &  \widehat{\phi}_1^{q-2}\\
\vdots           & \ddots            & \ddots                       & \ddots & \vdots                 \\  
\vdots           & \ddots            & \ddots                        & \ddots & \vdots                 \\
 \widehat{\phi}_1^{q-1} &  \dots       & \dots                      & \dots       &    1           \\
\end{array}
\right),
$$
where
\begin{equation}\label{eq:phi_1_hat}
\widehat{\phi}_1=\frac{\sum_{i=1}^n\sum_{\ell=2}^q \widehat{E}_{i,\ell} \widehat{E}_{i,\ell-1}}{\sum_{i=1}^n\sum_{\ell=1}^{q-1} \widehat{E}_{i,\ell}^2},
\end{equation}
where $\widehat{E}=(\widehat{E}_{i,\ell})_{1\leq i\leq n,1\leq \ell\leq q}$ is defined in (\ref{eq: Ehat}).
Then, Assumptions \ref{th2(vii)}, \ref{th2(viii)}, \ref{th2(ix)} and \ref{th2(x)} of Theorem \ref{th2} are valid.
\end{prop}

The proof of Proposition \ref{prop:cond_th_AR} is given in Section \ref{sec:proofs}. It is based on the following lemma.

\begin{lemma}\label{lem:AR:estim}
Assume that $(E_{1,t})_t$, $(E_{2,t})_t$, ..., $(E_{n,t})_t$ in Model (\ref{eq:model:matriciel}) are independent AR(1) processes satisfying:
$$
E_{i,t}-\phi_1 E_{i,t-1}=Z_{i,t},\; \forall i\in\{1,\dots,n\},
$$
where the $Z_{i,t}$'s are zero-mean i.i.d. Gaussian random variables with variance $\sigma^2$ and $|\phi_1|<1$. Let
$$
\widehat{\phi}_1=\frac{\sum_{i=1}^n\sum_{\ell=2}^q \widehat{E}_{i,\ell} \widehat{E}_{i,\ell-1}}{\sum_{i=1}^n\sum_{\ell=1}^{q-1} \widehat{E}_{i,\ell}^2},
$$
where $\widehat{E}=(\widehat{E}_{i,\ell})_{1\leq i\leq n,1\leq \ell\leq q}$ is defined in  (\ref{eq: Ehat}).
Then,
$$
\sqrt{nq_n}(\widehat{\phi}_1-\phi_1)=O_p(1),\textrm{ as } n\to\infty.
$$
\end{lemma}

Lemma \ref{lem:AR:estim} is proved in Section \ref{sec:proofs}. Its proof is based on Lemma \ref{lem:denom} in Section \ref{sec:lemmas}.



\section{Numerical experiments}\label{sec:num_exp}

\noindent The goal of this section  is twofold: $i)$ to provide sanity
checks for  our theoretical  results in  a well-controlled  framework; 
and  $ii)$ to investigate the robustness of our  estimator to some violations of the
assumptions of our theoretical results.  The  latter may reveal a  broader scope of
applicability for our method than the one guaranteed by the theoretical results.

We  investigate  $i)$  in  the  AR(1)  framework  presented  in
Section~\ref{subsec:AR1}.    Indeed,    all   assumptions    made   in
Theorems~\ref{th1}    and   \ref{th2}    can    be   specified    with
well-controllable  simulation  parameters  in   the  AR(1)  case  with
balanced design matrix $X$.

Point  $ii)$  aims  to  explore the  limitations  of  our  theoretical
framework  and assess  its robustness.   To this  end, we  propose two
numerical studies  relaxing some of  the assumptions of  our theorems:
first, we study  the effect of an unbalanced design  -- which violates
the sufficient condition of the irrepresentability condition \textbf{\rm (IC)} given in Proposition \ref{prop_IC} -- on the
sign-consistency;  and second, we  study the effect  of other types of 
dependence  than an AR(1).  

In all experiments, the performance are assessed in terms of sign-consistency.
In  other words,  we evaluate  the probability  for the  sign of
various  estimators  to  be  equal to  $\sign(\B)$.   We  compare  the
performance of three different estimators:
\begin{itemize}
\item $\widehat{\mathcal{B}}$ defined in (\ref{eq:lasso}),  which corresponds to the LASSO criterion applied
  to  the data  whitened with  the true  covariance matrix $\Sigma$;   
  we  call   this   estimator   \texttt{oracle}. Its theoretical properties are established in
  Theorem~\ref{th1}.
\item  $\widetilde{\mathcal{B}}$ defined in (\ref{eq:lasso_tilde}), which  corresponds  to the LASSO criterion
  applied  to the  data whitened  with  an  estimator of  the covariance matrix $\widehat{\Sigma}$; we  refer to this  estimator as
  \texttt{whitened-lasso}. Its theoretical properties are established in Theorem~\ref{th2}.
\item  the  LASSO  criterion applied to  the raw  data,  which we  call
  \texttt{raw-lasso}  hereafter. Its theoretical properties are established only in the univariate case in \cite{alquier:doukhan:2011}.
\end{itemize}

\subsection{AR(1) dependence structure with balanced one-way ANOVA}\label{sec:num:goodcase}

In this section,  we consider Model~\eqref{eq:model:matriciel} where
$X$ is the design matrix of  a one-way ANOVA with two balanced
groups.   Each  row  of  the random error  matrix  $E$  is distributed as a centered
Gaussian  random vector as  in
Equation~\eqref{eq:def_E} where the matrix $\Sigma$ is the  covariance matrix of an AR(1) process
defined in Section \ref{subsec:AR1}.

In this setting,  Assumptions \ref{th1(i)}, \ref{th1(ii)} and Condition~\ref{eq:irrep}
 of Theorem \ref{th1} are satisfied, see Propositions \ref{propcond} and \ref{prop_IC}.
The  three remaining  assumptions~\ref{th1(iii)}, \ref{th1(iv)}         and
\ref{assum:lambda_nq}  are  related   to  more  practical  quantities:
\ref{th1(iii)} controls  the sparsity level of  the problem, involving
$c_1$;  \ref{th1(iv)} basically controls  the signal-to-noise
ratio,  involving  $c_2$  and  \ref{assum:lambda_nq}  links
the  sample size $n$, $q$ and  the two
constants $c_1$, $c_2$, so that an appropriate range of penalty $\lambda$ exists for
having  a large  probability  of support  recovery. This  latter
assumption is  used in our  experiments to  tune the difficulty  of the
support recovery as follows: we consider different values of $n$, $q$, $c_1$, $c_2$
 and we choose  a sparsity level $|J|$  and a
minimal  magnitude  in  $\B$   such  that  Assumptions  \ref{th1(iii)}
 and \ref{th1(iv)} are fulfilled. Hence,
the  problem  difficulty is  essentially  driven  by the  validity  of
Assumption     \ref{assum:lambda_nq}    where     $q=o(n^{k})$    with
$c_1+c_2=1/2k$, and so by the relationship between $n$, $q$ and $k$.

We consider  a large  range of  sample sizes $n$  varying from  $10$ to
$1000$  and   three  different   values  for  
$q$ in  $\{10, 50, 1000  \}$.  The  constants $c_1$, $c_2$ are  chosen such
that      $c_1      +      c_2=1/2k$      with      $c_1=c_2$      and
$k$ in $\{1,2,4\}$. Additional   values   of  $c_1$ and $c_2$   have also been considered
and the corresponding results are available upon request. Finally,  we consider two values for
 the parameter $\phi_1$ appearing in the definition of the AR(1) process: $\phi_1\in\{0.5,  0.95\}$.  

Note that in this AR(1) setting with the estimator $\widehat{\phi}_1$ of $\phi_1$ defined in \eqref{eq:phi_1_hat} , all
the assumptions of  Theorem~\ref{th2} are  fulfilled,  see Proposition~\ref{prop:cond_th_AR}.   

The  frequencies  of support  recovery  for the  three  estimators
averaged    over  1000    replications    is     displayed    in
Figure~\ref{fig:ar1bal}.

\begin{figure}[htbp!]
\centering
\includegraphics[scale=1]{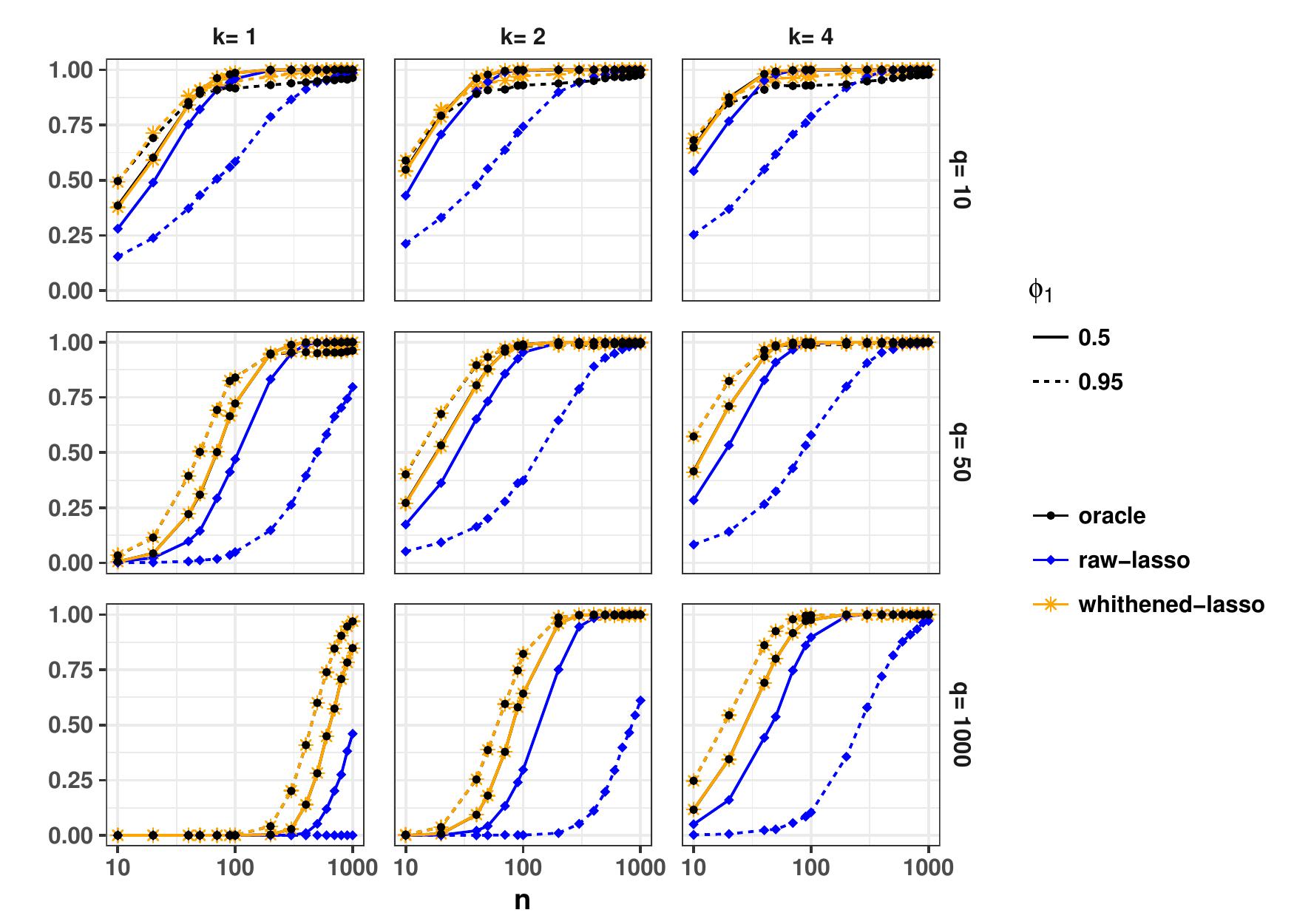}
\caption{Frequencies  of support  recovery  in a multivariate  one-way
  ANOVA model with two balanced groups and an AR(1) dependence.}
\label{fig:ar1bal}
\end{figure}

We observe from Figure~\ref{fig:ar1bal} that \texttt{whitened-lasso}   and \texttt{oracle}
have similar performance since $\phi_1$ is well estimated. These two approaches always exhibit better performance
than \texttt{raw-lasso}, especially when $\phi_1=0.95$. In this case, the sample size $n$ required to reach the same
performance  is indeed ten  time larger  for \texttt{raw-lasso} than for
\texttt{oracle} and \texttt{whitened-lasso}.

Finally, the performance  of all estimators  are altered when $n$ is too
small, especially  in situations
where the signal to noise ratio (SNR) is small  and the signal is not sparse enough,  these two
characteristics corresponding  to small values  of $k$.

\subsection{Robustness to unbalanced designs and correlated features}\label{sec:num:robustness}

The goal of this section is to study some particular design matrices
$X$ in Model~\eqref{eq:model:matriciel} that  may lead to violation of
the  Irrepresentability Condition~\ref{eq:irrep}.   

To    this    end,    we     consider    the    multivariate    linear
model~\eqref{eq:model:matriciel}  with   the  same   AR(1)  dependence
as the one considered in Section \ref{sec:num:goodcase}. Then, two 
different matrices $X$ are considered:
First,  an  one-way  ANOVA  model  with  two  unbalanced  groups  with
respective  sizes $n_1$ and $n_2$  such that  $n_1+n_2=n$;  and second,  a
multiple  regression model  with $p$ correlated  Gaussian predictors  such
that the rows of $X$ are i.i.d. $\mathcal{N}(0,\Sigma^X)$. 

For  the one-way ANOVA, violation
of \ref{eq:irrep} may occur when 
$r=n_1/n$    is   too    different    from   1/2,    as   stated    in
Proposition~\ref{prop_IC}.   For  the   regression  model,  we  choose for
$\Sigma^X$   a  $9\times   9$   matrix   ($p=9$) such  that   $\Sigma_{i,i}^X=1$,
$\Sigma_{i,j}^X = \rho$, when $i\neq j$. 
The other simulation parameters
are  fixed  as  in Section~\ref{sec:num:goodcase}.   

We  report  in
Figure \ref{fig:ic_violation} the results for the case where $q=1000$ and $k=2$
both for  unbalanced one-way  ANOVA (top  panels) and  regression with
correlated  predictors  (bottom  panels).    For  the  one-way  ANOVA,
 $r$  varies in  $\{0.4, 0.2,  0.1\}$. For  the regression
case, $\rho$ varies in $\{0.2,  0.6, 0.9\}$. In
both  cases, the  gray lines  correspond to  the  ideal
situation   (that  is,   either  unbalanced   or  uncorrelated) denoted \texttt{Ideal} in the legend of Figure \ref{fig:ic_violation}.   
The probability of support recovery is estimated over 1000 runs.

\begin{figure}[htbp!]
\centering
 \includegraphics{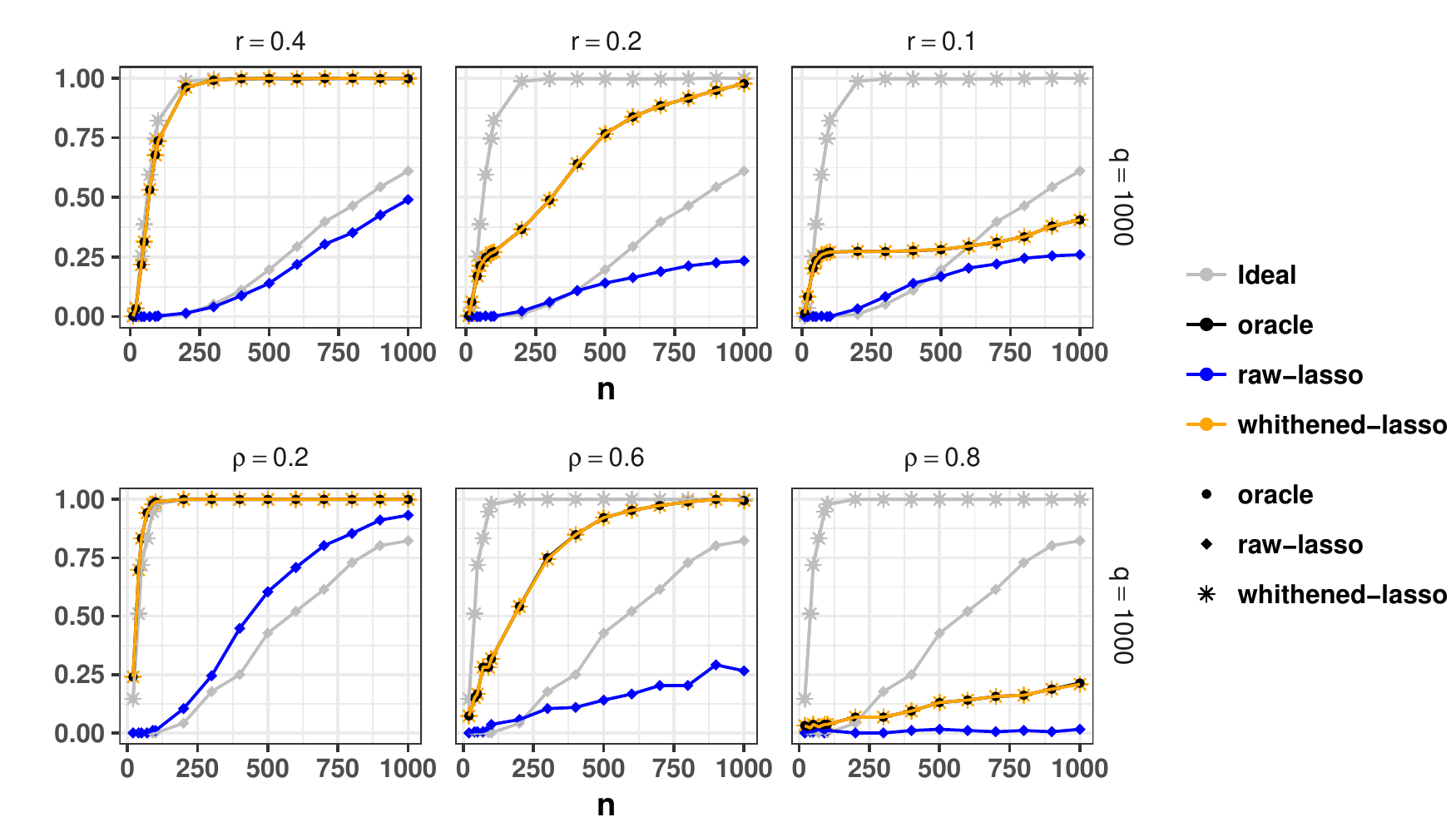}
\caption{Frequencies of support  recovery in  general linear  models
  with unbalanced designs: one-way ANOVA and regression.}
\label{fig:ic_violation}
\end{figure}

From  this figure,  we  note that  correlated  features or  unbalanced
designs deteriorate the  support recovery of all  estimators. This was
expected for these LASSO-based methods  which all suffer from the violation of
the  irrepresentability  condition~\ref{eq:irrep}.  However,  we  also
note that  \texttt{whitened-lasso} and  \texttt{oracle} have similar performance, which
means that the estimation of $\Sigma$ is
not altered, and that whitening always improves the support recovery.

\subsection{Robustness to more general autoregressive processes}

In  this section,  we consider  the  case where $X$ is the design matrix of a one-way ANOVA with two balanced
groups and  where $\Sigma$ is the covariance matrix of an AR($m$) process with $m$ in $\{5,10\}$.
Figure \ref{fig:arlong} displays the performance of the different estimators when $q=500$.
Here, for computing $\widehat{\Sigma}$ in \texttt{whitened-lasso}, the parameters $\phi_1,\dots,\phi_m$ of the AR($m$) process are estimated as follows.
They are obtained by averaging over the $n$ rows of $\widehat{E}$ defined in (\ref{eq: Ehat}) the estimations
$\widehat{\phi}_1^{(i)},\dots,\widehat{\phi}_m^{(i)}$ obtained for the $i$th row of $\widehat{E}$ by using standard estimation approaches for AR processes described in \cite{Brockwell:1990}.
As previously, we observe from this figure that \texttt{whitened-lasso} and \texttt{oracle} have better performance than \texttt{raw-lasso}.

\begin{figure}[htbp!]
\centering
\includegraphics[scale=1]{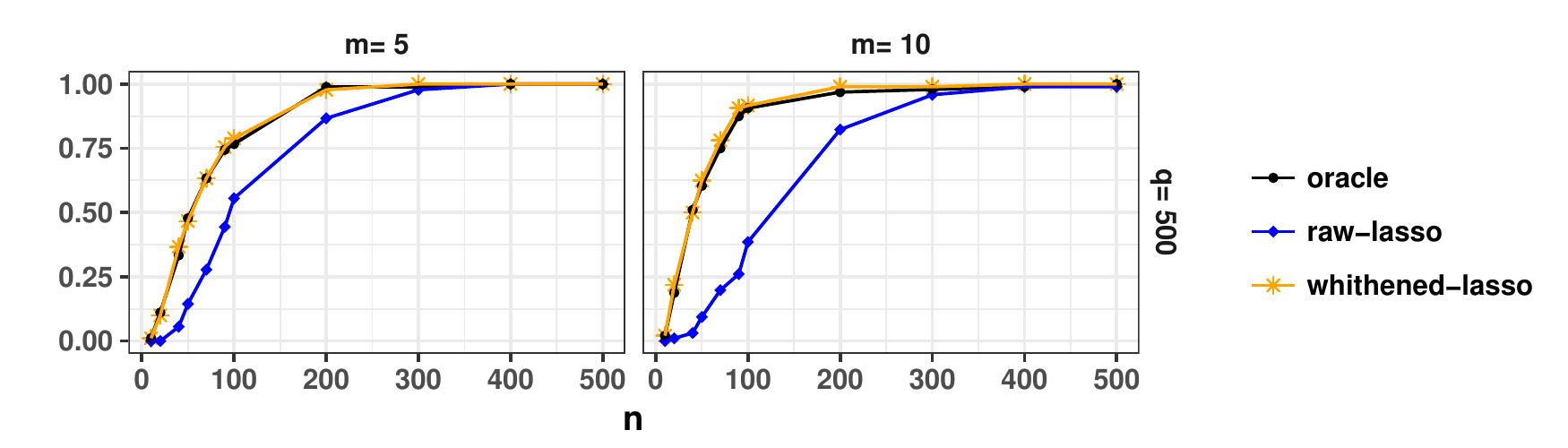}
\caption{Frequencies  of support  recovery in one-way ANOVA with AR($m$) covariance matrix.}
\label{fig:arlong}
\end{figure}


\section{Proofs}\label{sec:proofs}

\begin{proof}[Proof of Proposition \ref{prop1}] 
For a fixed nonnegative $\lambda$, by \eqref{eq:lasso},
$$
\widehat{\mathcal{B}}=\widehat{\mathcal{B}}(\lambda)=\textrm{Argmin}_{\mathcal{B}}\left\{\|\mathcal{Y}-\mathcal{X}\mathcal{B}\|_2^2+\lambda\|\mathcal{B}\|_1\right\}.
$$
Denoting  $\widehat{u}=\widehat{\mathcal{B}}-\mathcal{B}$, we get
\begin{align*}
\|\mathcal{Y}-\mathcal{X}\widehat{\mathcal{B}}\|_2^2+\lambda\|\widehat{\mathcal{B}}\|_1
&=\|\mathcal{X}\mathcal{B}+\mathcal{E}-\mathcal{X}\widehat{\mathcal{B}}\|_2^2+\lambda \|\widehat{u}+\mathcal{B}\|_1
=\|\mathcal{E}-\mathcal{X}\widehat{u}\|_2^2+\lambda \|\widehat{u}+\B\|_1\\
&=\|\E\|_2^2-2\widehat{u}'\X'\E+\widehat{u}'\X'\X\widehat{u}+\lambda \|\widehat{u}+\B\|_1.
\end{align*}
Thus,
$$
\widehat{u}=\textrm{Argmin}_u\; V(u),
$$
where
$$
V(u)=-2 (\sqrt{nq}u)' W+(\sqrt{nq}u)' C (\sqrt{nq}u)+\lambda \|u+\B\|_1.
$$
Since the first derivative of $V$ with respect to $u$ is equal to
$$
2\sqrt{nq}\big(C(\sqrt{nq}u)-W\big)+\lambda\,\sign(u+\B),$$
$\widehat{u}$ satisfies
\[
C_{J,J}(\sqrt{nq}\widehat{u}_J)-W_J=-\frac{\lambda}{2\sqrt{nq}}\sign(\widehat{u}_J+\B_J)
=-\frac{\lambda}{2\sqrt{nq}}\sign(\widehat{\B}_J), \textrm{\, if \,}\widehat{u}_J+\B_J=\widehat{\B}_J\neq 0
\]
and
\[
\big|C_{J^c,J}(\sqrt{nq}\widehat{u}_J)-W_{J^c}\big|\leq\frac{\lambda}{2\sqrt{nq}}.\]
%
Note that, if $|\widehat{u}_J|  < |\B_J|$, then $\widehat{\B}_J\neq 0$ and $\textrm{sign}(\widehat{\B}_J)=\textrm{sign}(\B_J)$.

Let us now prove that when $A_n$ and $B_n$, defined in (\ref{eq:An}) and (\ref{eq:Bn}), are satisfied then there exists 
$\widehat{u}$ satisfying:
\begin{align}
C_{J,J}(\sqrt{nq}\widehat{u}_J)-W_J&=-\frac{\lambda}{2\sqrt{nq}}\textrm{sign}(\B_J),\label{eq:support1}\\
|\widehat{u}_J| & < |\B_J|,\label{eq:support2}\\
\big|C_{J^c,J}(\sqrt{nq}\widehat{u}_J)-W_{J^c}\big|&\leq\frac{\lambda}{2\sqrt{nq}}\label{eq:support3}.
\end{align}
Note that $A_n$ implies:
\begin{equation}\label{eq:u1}
\sqrt{nq}\left(-|\B_J|+\frac{\lambda}{2nq}(C_{J,J})^{-1}\textrm{sign}(\B_J)\right)<(C_{J,J})^{-1}W_J<\sqrt{nq}\left(|\B_J|+\frac{\lambda}{2nq}(C_{J,J})^{-1}\textrm{sign}(\B_J)\right).
\end{equation}
By denoting
\begin{equation}\label{eq:u_J_hat}
\widehat{u}_J=\frac{1}{\sqrt{nq}}(C_{J,J})^{-1}W_J-\frac{\lambda}{2nq}(C_{J,J})^{-1}\textrm{sign}(\B_J),
\end{equation}
we obtain from (\ref{eq:u1}) that (\ref{eq:support1}) and (\ref{eq:support2}) hold.
Note that $B_n$ implies:
\begin{multline*}
-\frac{\lambda}{2\sqrt{nq}}\left(\mathbf{1}-C_{J^c,J}(C_{J,J})^{-1}\textrm{sign}(\B_J)\right)\\
\leq C_{J^c,J}(C_{J,J})^{-1}W_J-W_{J^c}\leq\frac{\lambda}{2\sqrt{nq}}\left(\mathbf{1}+C_{J^c,J}(C_{J,J})^{-1}
\textrm{sign}(\B_J)\right).
\end{multline*}
Hence,
$$
\left|C_{J^c,J}\left((C_{J,J})^{-1}W_J-\frac{\lambda}{2\sqrt{nq}}(C_{J,J})^{-1}\textrm{sign}(\B_J)\right)-W_{J^c}\right|\leq\frac{\lambda}{2\sqrt{nq}},
$$
which is (\ref{eq:support3}) by (\ref{eq:u_J_hat}). This concludes the proof.
\end{proof}

\begin{proof}[Proof of Theorem \ref{th1}]
By Proposition \ref{prop1},
$$
\PP\left(\textrm{sign}(\widehat{\B}(\lambda))=\textrm{sign}(\B)\right)\geq\PP(A_n\cap B_n)
=1-\PP(A_n^c\cup B_n^c)\geq 1-\PP(A_n^c)-\PP(B_n^c),
$$
where $A_n$ and $B_n$ are defined in (\ref{eq:An}) and (\ref{eq:Bn}). It is thus enough to prove that $\PP(A_n^c)$ and $\PP(B_n^c)$ tend to zero as $n$ tends to infinity.

By definition of $A_n$,
\begin{align}\label{eq:P_Anc}
\PP(A_n^c)&=\PP\left(\left|(C_{J,J})^{-1}W_J\right|\geq\sqrt{nq}\left(|\B_J|-\frac{\lambda}{2nq}|(C_{J,J})^{-1}\textrm{sign}(\B_J)|\right)\right)\nonumber\\
&\leq\sup_{j\in J} \PP\left(|\xi_j|\geq\sqrt{nq}\left(|\B_j|-\frac{\lambda}{2nq}|b_j|\right)\right),
\end{align}
where 
$$
\xi=(\xi_j)_{j\in J}=(C_{J,J})^{-1}W_J=\frac{1}{\sqrt{nq}}(C_{J,J})^{-1}(\X_{\bullet,J})'\E =: H_A\;\E,
$$
and
$$
b=(b_j)_{j\in J}=(C_{J,J})^{-1}\textrm{sign}(\B_J).
$$
By definition of $B_n$ and \ref{eq:irrep},
\begin{align}\label{eq:P_Bnc}
\PP(B_n^c)&=\PP\left(\left|C_{J^c,J}(C_{J,J})^{-1}W_J-W_{J^c}\right|>\frac{\lambda}{2\sqrt{nq}}\left(\mathbf{1}-\left|C_{J^c,J}(C_{J,J})^{-1}\textrm{sign}(\B_J)\right|\right)\right)\nonumber\\
&\leq \PP\left(\left|C_{J^c,J}(C_{J,J})^{-1}W_J-W_{J^c}\right|>\frac{\lambda}{2\sqrt{nq}}\eta\right)\nonumber\\
&\leq \sup_{j\in J^c} \PP\left(|\zeta_j|>\frac{\lambda}{2\sqrt{nq}}\eta\right),
\end{align}
where
$$
\zeta=(\zeta_j)_{j\in J^c}=C_{J^c,J}(C_{J,J})^{-1}W_J-W_{J^c}=\frac{1}{\sqrt{nq}}\big(C_{J^c,J}(C_{J,J})^{-1}(\X_{\bullet,J})'-(\X_{\bullet,J^c})'\big)\E
 =: H_B\;\E.
 $$
Note that, for all $j$ in $J$,
$$
|b_j|\leq\sum_{j\in J} |b_j|\leq \sqrt{|J|}\left(\sum_{j\in J} b_j^2\right)^{1/2}=\sqrt{|J|}\|b\|_2.
$$ 
Moreover,
$$
\|b\|_2=\|(C_{J,J})^{-1}\textrm{sign}(\B_J)\|_2\leq \|(C_{J,J})^{-1}\|_2\sqrt{|J|}:=\lambda_{\textrm{max}}((C_{J,J})^{-1})\sqrt{|J|}, 
$$
where $\lambda_{\textrm{max}}(A)$ denotes the largest eigenvalue of the matrix $A$. Observe that
\begin{equation}\label{eq:lambda_max_CJJ-1}
\lambda_{\textrm{max}}((C_{J,J})^{-1})=\frac{1}{\lambdaMin(C_{J,J})}=\frac{q}{\lambdaMin((\mathcal{X}'\mathcal{X})_{J,J})/n}\leq \frac{q}{M_2},
\end{equation}
by Assumption~\ref{th1(ii)} of Theorem~\ref{th1}. Thus, for all $j$ in $J$,
\begin{equation}\label{eq:bj}
|b_j|\leq \frac{q|J|}{M_2}.
\end{equation}
By Assumption~\ref{th1(iv)} of Theorem~\ref{th1}, we get thus that for all $j$ in $J$,
\begin{equation}\label{eq:minoration_terme_A_th1}
\sqrt{nq}\left(|\B_j|-\frac{\lambda}{2nq}\left|\left((C_{J,J})^{-1}\textrm{sign}(\B_J)\right)_j\right|\right)=\sqrt{nq}\left(|\B_j|-\frac{\lambda}{2nq}|b_j|\right)
\geq \sqrt{nq}\left(M_3 q^{-c_2} -\frac{\lambda q |J|}{2nq M_2}\right).
\end{equation}
Thus,
\begin{equation}\label{eq:Anc_bis}
\PP(A_n^c)\leq\sup_{j\in J} \PP\left(|\xi_j|\geq\sqrt{nq}\left(M_3 q^{-c_2} -\frac{\lambda q |J|}{2nq M_2}\right)\right).
\end{equation}
Since $\E$ is a centered Gaussian random vector having a covariance matrix equal to identity,
$\xi= H_A\;\E$ is a centered Gaussian random vector with a covariance matrix equal to:
$$
H_A H_A'=\frac{1}{nq}(C_{J,J})^{-1}(\X_{\bullet,J})' \X_{\bullet,J}(C_{J,J})^{-1}=(C_{J,J})^{-1}.
$$
Hence, by (\ref{eq:lambda_max_CJJ-1}), we get that for all $j$ in $J$,
$$
\Var(\xi_j)=\left((C_{J,J})^{-1}\right)_{jj}\leq\lambda_{\textrm{max}}(C_{J,J}^{-1})\leq \frac{q}{M_2}.
$$
Thus,
$$
\PP\left(|\xi_j|\geq\sqrt{nq}\left(M_3 q^{-c_2} -\frac{\lambda q |J|}{2nq M_2}\right)\right)
\leq \PP\left(|Z|\geq\frac{\sqrt{M_2}}{\sqrt{q}}\left(M_3 q^{-c_2}\sqrt{nq}-\frac{\lambda q |J|}{2\sqrt{nq} M_2}\right)\right),
$$ 
where $Z$ is a standard Gaussian random variable.
By Chernoff inequality, we thus obtain that for all $j$ in $J$,
$$
\PP\left(|\xi_j|\geq\sqrt{nq}\left(M_3 q^{-c_2} -\frac{\lambda q |J|}{2nq M_2}\right)\right)
\leq 2\exp\left(-\frac{M_2}{2q}\left\{M_3 q^{-c_2}\sqrt{nq}-\frac{\lambda q |J|}{2\sqrt{nq} M_2}\right\}^2\right).
$$
By Assumption~\ref{th1(iii)} of Theorem~\ref{th1}, we get that under the last condition of \ref{assum:lambda_nq},
\begin{equation}\label{eq:lambda_J_sqrt_nq}
\frac{\lambda q |J|}{\sqrt{nq}}=o\left(q^{-c_2}\sqrt{nq}\right), \textrm{ as } n\to\infty.
\end{equation}
Thus,
\begin{equation}\label{eq:Anc:limit}
\PP(A_n^c)\to 0,\; \textrm{ as } n\to\infty.
\end{equation}

Let us now bound $\PP(B_n^c)$. Observe that
$\zeta= H_B\;\E$ is a centered Gaussian random vector with a covariance matrix equal to:
\begin{align*}
H_B H_B' &=\frac{1}{nq}(C_{J^c,J}(C_{J,J})^{-1}(\X_{\bullet,J})'-\X_{\bullet,J^c}')(\X_{\bullet,J}(C_{J,J})^{-1}C_{J,J^c}-\X_{\bullet,J^c})\\
&=C_{J^c,J^c}-C_{J^c,J}(C_{J,J})^{-1}C_{J,J^c}=\frac{1}{nq} (\X_{\bullet,J^c})'\left(\textrm{Id}_{\rset^{nq}}-\X_{\bullet,J}((\X_{\bullet,J})'\X_{\bullet,J})^{-1}(\X_{\bullet,J})'\right)\X_{\bullet,J^c}\\
&=\frac{1}{nq} (\X_{\bullet,J^c})' \, \left(\textrm{Id}_{\rset^{nq}}-\Pi_{\textrm{Im}(\X_{\bullet,J})}\right) \, \X_{\bullet,J^c},
\end{align*}
where $\Pi_{\textrm{Im}(\X_{\bullet,J})}$ denotes the orthogonal projection onto the column space of $\X_{\bullet,J}$.
Note that, for all $j$ in $J^c$,
\begin{align*}
\Var(\zeta_j)&=\frac{1}{nq}\left((\X_{\bullet,J^c})'\; \left(\textrm{Id}_{\rset^{nq}}-\Pi_{\textrm{Im}(\X_{\bullet,J})}\right)\; \X_{\bullet,J^c}\right)_{jj}\\
&=\frac{1}{nq}\left((\X_{\bullet,J^c})' \X_{\bullet,J^c}\right)_{jj}-\frac{1}{nq}\left((\X_{\bullet,J^c})'\;\Pi_{\textrm{Im}(\X_{\bullet,J})}\X_{\bullet,J^c}\right)_{jj}\\
&\leq\frac{1}{nq}\left((\X_{\bullet,J^c})' \X_{\bullet,J^c}\right)_{jj}\leq \frac{M_1}{q},
\end{align*}
where the inequalities come from Lemma~\ref{lem:proj} and Assumption~\ref{th1(i)} of Theorem~\ref{th1}.
Thus, for all $j$ in $J^c$,
$$
\PP\left(|\zeta_j|>\frac{\lambda}{2\sqrt{nq}}\eta\right)\leq \PP\left(|Z|>\frac{\lambda\sqrt{q}}{2\sqrt{M_1}\sqrt{nq}}\eta\right),
$$
where $Z$ is a standard Gaussian random variable. By Chernoff inequality, for  all $j$ in $J^c$,
$$
\PP\left(|\zeta_j|>\frac{\lambda}{2\sqrt{nq}}\eta\right)\leq 2\exp\left\{-\frac12\left(\frac{\lambda}{2\sqrt{M_1}\sqrt{n}}\eta\right)^2\right\}.
$$
Hence, under the following assumption
$$
\frac{\lambda}{\sqrt{n}}\to\infty,
$$
which is the second condition of  \ref{assum:lambda_nq},
\begin{equation}\label{eq:Bnc:limit}
\PP(B_n^c)\to 0,\; \textrm{ as } n\to\infty.
\end{equation}
\end{proof}

\begin{proof}[Proof of Proposition \ref{propcond}]
Let us first prove that \ref{cond1} and \ref{cond3} imply \ref{th1(i)}.
For $j\in\{1,\ldots,pq\}$, by considering the Euclidian division of $j-1$ by $p$ given by $(j-1)=pk_j+r_j$, we observe that
\begin{align*}
(\mathcal{X}_{\bullet,j})'\mathcal{X}_{\bullet,j}&=(((\Sigma^{-1/2})' \otimes X)_{\bullet,j})'((\Sigma^{-1/2})' \otimes X)_{\bullet,j}\\
&=((\Sigma^{-1/2}) \otimes X')_{j,\bullet})((\Sigma^{-1/2})' \otimes X)_{\bullet,j}\\
&=((\Sigma^{-1/2})_{k_j+1,\bullet} \otimes (X_{\bullet,r_j+1})')(((\Sigma^{-1/2})_{\bullet,k_j+1})' \otimes X_{\bullet,r_j+1})\\
&=(\Sigma^{-1/2})_{k_j+1,\bullet}((\Sigma^{-1/2})_{\bullet,k_j+1})' \otimes (X_{\bullet,r_j+1})'X_{\bullet,r_j+1}\\
&=(\Sigma^{-1})_{k_j+1,k_j+1} \otimes(X_{\bullet,r_j+1})'X_{\bullet,r_j+1}\\
&=(\Sigma^{-1})_{k_j+1,k_j+1}(X_{\bullet,r_j+1})'X_{\bullet,r_j+1}.
\end{align*}
Hence, using \ref{cond1}, we get that for all $j$ in $\{1,\ldots,pq\}$,
\begin{align*}
\frac{1}{n}(\mathcal{X}_{\bullet,j})'\mathcal{X}_{\bullet,j}&\leq M'_1(\Sigma^{-1})_{k_j+1,k_j+1}
\leq M'_1\sup_{k \in \{0,\ldots,q-1\}}((\Sigma^{-1})_{k+1,k+1})\\
&  \leq M'_1\lambdaMax(\Sigma^{-1})\leq M'_1 m_1,
\end{align*}
where the last inequality comes from~\ref{cond3}, which gives (A1).

Let us now prove that \ref{cond2} and \ref{cond4} imply \ref{th1(ii)}.
%
%
Note that
\begin{align*}
(\mathcal{X}'\mathcal{X})_{J,J}&=(((\Sigma^{-1/2})' \otimes X)'((\Sigma^{-1/2})' \otimes X))_{J,J}\\
&=(\Sigma^{-1/2}(\Sigma^{-1/2})'\otimes X'X)_{J,J}\\
&=(\Sigma^{-1} \otimes X'X)_{J,J}.
\end{align*}
Then, by Theorem 4.3.15 of \cite{Horn:Johnson:1985}, 
\begin{align}
\lambda_{\min}((\mathcal{X}'\mathcal{X})_{J,J})&=\lambda_{\min}((\Sigma^{-1} \otimes X'X)_{J,J}) \nonumber\\
& \geq \lambda_{\min}(\Sigma^{-1} \otimes X'X) \nonumber\\
&=\lambda_{\min}(X'X)\lambda_{\min}(\Sigma^{-1}).\label{lmbgreat}
\end{align}
Finally, by using Conditions~\ref{cond2}~and~\ref{cond4}, we obtain
\[
\frac{1}{n}\lambda_{\min}(\mathcal{X}'\mathcal{X})_{J,J}\geq \frac{1}{n}\lambda_{\min}(X'X)\lambda_{\min}(\Sigma^{-1})\geq M'_2m_2,
\]
which gives \ref{th1(ii)}.
\end{proof}

\begin{proof}[Proof of Theorem \ref{th2}]
By Proposition \ref{prop2},
$$
\PP\left(\textrm{sign}(\widetilde{\mathcal{B}}(\lambda))=\textrm{sign}(\mathcal{B})\right)\geq\PP(\widetilde{A}_n\cap\widetilde{B}_n)
=1-\PP(\widetilde{A}_n^c\cup \widetilde{B}_n^c)\geq 1-\PP(\widetilde{A}_n^c)-\PP(\widetilde{B}_n^c),
$$
where $\widetilde{A}_n$ and $\widetilde{B}_n$ are defined in (\ref{eq:An_tilde}) and (\ref{eq:Bn_tilde}).
By definition of $\widetilde{A}_n$, we get
$$
\PP(\widetilde{A}_n^c)=\PP\left(\left\{\left|(\widetilde{C}_{J,J})^{-1}\widetilde{W}_J\right|\geq\sqrt{nq}\left(|\mathcal{B}_J|
-\frac{\lambda}{2nq}|(\widetilde{C}_{J,J})^{-1}\textrm{sign}(\mathcal{B}_J)|\right)\right\}\right).
$$
Observing that
\begin{align*}
(\widetilde{C}_{J,J})^{-1}\widetilde{W}_J&=(C_{J,J})^{-1}W_J+(C_{J,J})^{-1}\left(\widetilde{W}_J-W_J\right)\\
&+\left((\widetilde{C}_{J,J})^{-1}-(C_{J,J})^{-1}\right)W_J\\
&+\left((\widetilde{C}_{J,J})^{-1}-(C_{J,J})^{-1}\right)\left(\widetilde{W}_J-W_J\right),
\end{align*}
$$
(\widetilde{C}_{J,J})^{-1}\textrm{sign}(\mathcal{B}_J)=(C_{J,J})^{-1}\textrm{sign}(\mathcal{B}_J)+\left((\widetilde{C}_{J,J})^{-1}-(C_{J,J})^{-1}\right)
\textrm{sign}(\mathcal{B}_J),
$$
and using the triangle inequality, we obtain that
\begin{align}\label{eq:An_tilde_bound}
&\PP(\widetilde{A}_n^c)\leq\PP\left(\left\{\left|(C_{J,J})^{-1}W_J\right|\geq\frac{\sqrt{nq}}{5}\left(|\mathcal{B}_J|
-\frac{\lambda}{2nq}\left|(C_{J,J})^{-1}\textrm{sign}(\mathcal{B}_J)\right|\right)\right\}\right)\nonumber\\
&+\PP\left(\left\{\left|(C_{J,J})^{-1}\left(\widetilde{W}_J-W_J\right)\right|
\geq\frac{\sqrt{nq}}{5}\left(|\mathcal{B}_J|-\frac{\lambda}{2nq}\left|(C_{J,J})^{-1}\textrm{sign}(\mathcal{B}_J)\right|\right)\right\}\right)\nonumber\\
&+\PP\left(\left\{\left|\left((\widetilde{C}_{J,J})^{-1}-(C_{J,J})^{-1}\right)W_J\right|
\geq\frac{\sqrt{nq}}{5}\left(|\mathcal{B}_J|-\frac{\lambda}{2nq}\left|(C_{J,J})^{-1}\textrm{sign}(\mathcal{B}_J)\right|\right)\right\}\right)\nonumber\\
&+\PP\left(\left\{\left|\left((\widetilde{C}_{J,J})^{-1}-(C_{J,J})^{-1}\right)\left(\widetilde{W}_J-W_J\right)\right|
\geq\frac{\sqrt{nq}}{5}\left(|\mathcal{B}_J|-\frac{\lambda}{2nq}\left|(C_{J,J})^{-1}\textrm{sign}(\mathcal{B}_J)\right|\right)\right\}\right)\nonumber\\
&+\PP\left(\left\{\frac{\lambda}{2\sqrt{nq}}\left|\left((\widetilde{C}_{J,J})^{-1}-(C_{J,J})^{-1}\right)\textrm{sign}(\mathcal{B}_J)\right|\geq
\frac{\sqrt{nq}}{5}\left(|\mathcal{B}_J|-\frac{\lambda}{2nq}\left|(C_{J,J})^{-1}\textrm{sign}(\mathcal{B}_J)\right|\right)\right\}\right).\nonumber\\
\end{align}
The first term in the r.h.s of (\ref{eq:An_tilde_bound}) tends to 0 by the definition of ${A}_n^c$ and (\ref{eq:Anc:limit}).
By (\ref{eq:minoration_terme_A_th1}), the last term of (\ref{eq:An_tilde_bound}) satisfies, for all $j\in J$:
\begin{align*}
&\PP\left(\left|\left((\widetilde{C}_{J,J})^{-1}-(C_{J,J})^{-1}\right)\textrm{sign}(\mathcal{B}_J)\right|\geq
\frac{2 nq}{5\lambda}\left(|\mathcal{B}_J|-\frac{\lambda}{2nq}\left|(C_{J,J})^{-1}\textrm{sign}(\mathcal{B}_J)\right|\right)\right)\\
&\leq \PP\left(\left|\left(\left((\widetilde{C}_{J,J})^{-1}-(C_{J,J})^{-1}\right)\textrm{sign}(\mathcal{B}_J)\right)_j\right|\geq
\frac{2 nq}{5\lambda}\left(M_3 q^{-c_2} -\frac{\lambda q|J|}{2nq M_2}\right)\right).
\end{align*}
Let $U=(\widetilde{C}_{J,J})^{-1}-(C_{J,J})^{-1}$ and $s=\textrm{sign}(\mathcal{B}_J)$ then for all $j$ in $J$:
\begin{equation}\label{eq:U_s_j}
|(U s)_j|=\left|\sum_{k\in J} U_{jk} s_k\right|\leq \sqrt{|J|}\|U\|_{2}.
\end{equation}
We focus on
\begin{align*}
&\|(\widetilde{C}_{J,J})^{-1}-(C_{J,J})^{-1}\|_2=\|(\widetilde{C}_{J,J})^{-1}(C_{J,J}-\widetilde{C}_{J,J})(C_{J,J})^{-1}\|_{2}
\leq \|(\widetilde{C}_{J,J})^{-1}\|_{2}\;\|C_{J,J}-\widetilde{C}_{J,J}\|_{2}\;\|(C_{J,J})^{-1}\|_{2}\\
&\leq \frac{\rho(C_{J,J}-\widetilde{C}_{J,J})}{\lambdaMin(\widetilde{C}_{J,J})\lambdaMin({C}_{JJ})}
\leq \frac{\rho(C_{J,J}-\widetilde{C}_{J,J})}{\lambdaMin(\widetilde{C}_{J,J}) (M_2/q)},
\end{align*}
where the last inequality comes from Assumption \ref{th1(ii)} of Theorem \ref{th1}, which gives that
\begin{equation}\label{eq:CJJ-1:bound}
\|(C_{J,J})^{-1}\|_{2}\leq \frac{q}{M_2}.
\end{equation}
Using Theorem 4.3.15 of \cite{Horn:Johnson:1985}, we get
$$
\|(\widetilde{C}_{J,J})^{-1}-(C_{J,J})^{-1}\|_2\leq \frac{q\rho(C-\widetilde{C})}{\lambda_{\textrm{min}}(\widetilde{C}) M_2}.
$$
By definition of $C$ and $\widetilde{C}$ given in (\ref{eq:C:J}) and (\ref{eq:C_tilde}), respectively, we get
\begin{equation}\label{eq:C_C_tilde}
C=\frac{\Sigma^{-1}\otimes (X'X)}{nq} \textrm{ and } \widetilde{C}=\frac{\widehat{\Sigma}^{-1}\otimes (X'X)}{nq}.
\end{equation}
By using that the eigenvalues of the Kronecker product of two matrices is equal to the product of the eigenvalues of the two matrices, we obtain
\begin{align}\label{eq:norme2_diff_CJJ}
\|(\widetilde{C}_{J,J})^{-1}-(C_{J,J})^{-1}\|_{2}
&\leq\frac{\rho(\Sigma^{-1}-\widehat{\Sigma}^{-1})\lambda_{\textrm{max}}((X'X)/n)q}{\lambda_{\textrm{min}}(\widehat{\Sigma}^{-1}) 
\lambda_{\textrm{min}}((X'X)/n) M_2}\leq \frac{\rho(\Sigma^{-1}-\widehat{\Sigma}^{-1})\lambda_{\textrm{max}}(\widehat{\Sigma}) \lambda_{\textrm{max}}((X'X)/n)q}
{\lambda_{\textrm{min}}((X'X)/n) M_2}\nonumber\\
&\leq\frac{\rho(\Sigma^{-1}-\widehat{\Sigma}^{-1})\left(\rho(\widehat{\Sigma}-\Sigma)+\lambda_{\textrm{max}}(\Sigma)\right) \lambda_{\textrm{max}}((X'X)/n)q}
{\lambda_{\textrm{min}}((X'X)/n) M_2},
\end{align}
where the last inequality follows from Theorem 4.3.1 of \cite{Horn:Johnson:1985}. Thus, by Assumptions \ref{th2(v)}, \ref{th2(vi)}, \ref{th2(viii)},  \ref{th2(ix)} and \ref{th2(x)}, we get that
\begin{equation}\label{eq:norm2_Ctilde_JJ-1_-C_JJ-1}
\|(\widetilde{C}_{J,J})^{-1}-(C_{J,J})^{-1}\|_{2}=O_P(q(nq)^{-1/2}), \textrm{ as } n\to\infty.
\end{equation}
%
Hence, by (\ref{eq:U_s_j}), we get for all $j$ in $J$ that
\begin{align*}
&\PP\left(\left|\left(\left((\widetilde{C}_{J,J})^{-1}-(C_{J,J})^{-1}\right)\textrm{sign}(\mathcal{B}_J)\right)_j\right|\geq
\frac{2 nq}{5\lambda}\left(M_3 q^{-c_2} -\frac{\lambda q|J|}{2nq M_2}\right)\right)\\
&\leq\PP\left(\sqrt{|J|}\;\|(\widetilde{C}_{J,J})^{-1}-(C_{J,J})^{-1}\|_{2}\geq
\frac{2 \sqrt{nq}}{5\lambda}\left(M_3 q^{-c_2}\sqrt{nq} -\frac{\lambda q|J|}{2\sqrt{nq} M_2}\right)\right).
\end{align*}
By (\ref{eq:lambda_J_sqrt_nq}), (\ref{eq:norm2_Ctilde_JJ-1_-C_JJ-1}) and \ref{th1(iii)}, it is enough to prove that
$$
\PP\left(q^{c_1/2}q(nq)^{-1/2}\geq \frac{nq}{\lambda}q^{-c_2}\right)\to 0, \textrm{ as } n\to\infty.
$$
By the last condition of \ref{assum:lambda_nq},
$$
\frac{\frac{nq}{\lambda}q^{-c_2}}{q^{1+c_1}}\to\infty, \textrm{ as } n\to\infty,
$$
and the result follows since $n$ tends to infinity.
Hence, the last term of (\ref{eq:An_tilde_bound}) tends to zero as $n$ tends to infinity.

Let us now study the second term in the r.h.s of (\ref{eq:An_tilde_bound}).
\begin{align}\label{eq:diff_Wn}
&\widetilde{W}_J-W_J=\frac{1}{\sqrt{nq}}\left(\left(\widetilde{\mathcal{X}}'\widetilde{\mathcal{E}}\right)_J-\left(\mathcal{X}'\mathcal{E}\right)_J\right)
=\frac{1}{\sqrt{nq}}\left(\widetilde{\mathcal{X}}'\widetilde{\mathcal{E}}-\mathcal{X}'\mathcal{E}\right)_J\nonumber\\
&=\frac{1}{\sqrt{nq}}\left[\left(\widehat{\Sigma}^{-1/2}\otimes X'\right)\left((\widehat{\Sigma}^{-1/2})'\otimes \textrm{Id}_{\rset^n}\right)\textrm{Vec}(E)
-\left(\Sigma^{-1/2}\otimes X'\right)\left((\Sigma^{-1/2})'\otimes \textrm{Id}_{\rset^n}\right)\textrm{Vec}(E)\right]_J\nonumber\\
&=\frac{1}{\sqrt{nq}}\left[\left\{\left(\widehat{\Sigma}^{-1}-\Sigma^{-1}\right)\otimes X'\right\}\textrm{Vec}(E)\right]_J
\stackrel{d}{=}AZ,
\end{align}
where $Z$ is a centered Gaussian random vector having a covariance matrix equal to identity and
\begin{equation}\label{eq:def_A}
A=\frac{1}{\sqrt{nq}}\left[\left\{\left(\widehat{\Sigma}^{-1}-\Sigma^{-1}\right)\otimes X'\right\}\left\{(\Sigma^{1/2})'\otimes\textrm{Id}_{\rset^n}\right\}\right]_{J,\bullet}.
\end{equation}
By Cauchy-Schwarz inequality, we get for all $K\times nq$ matrix $B$, and all $nq\times 1$ vector $U$ that for all $k$ in $\{1,\dots,K\}$,
\begin{equation}\label{eq:Cauchy-Schwarz}
\left|(BU)_k\right|=\left|\sum_{\ell=1}^{nq} B_{k,\ell} U_\ell\right|\leq \|B\|_2\;\|U\|_2.
\end{equation}
Thus, for all $j$ in $J$, for all $\gamma$ in $\rset$ and all $|J|\times |J|$ matrix $D$,
\begin{equation}\label{eq:B_WJ-WJtilde}
\PP\left(\left|\left(D\left(\widetilde{W}_J-W_J\right)\right)_j\right|\geq \gamma\right)
=\PP\left(\left|\left(DAZ\right)_j\right|\geq \gamma\right)
\leq\PP\left(\|D\|_2\;\|A\|_2 \|Z\|_2\geq \gamma\right),
\end{equation}
where $A$ is defined in (\ref{eq:def_A}) and $Z$ is a centered Gaussian random vector having a covariance matrix equal to identity. Hence, for all $j$ in $J$,
\begin{align*}
&\PP\left(\left|\left((C_{J,J})^{-1}\left(\widetilde{W}_J-W_J\right)\right)_j\right|
\geq\frac{\sqrt{nq}}{5}\left(|\mathcal{B}_j|-\frac{\lambda}{2nq}\left|\left((C_{J,J})^{-1}\textrm{sign}(\mathcal{B}_J)\right)_j\right|\right)\right)\\
&\leq \PP\left(\|(C_{J,J})^{-1}\|_2\|A\|_2\|Z\|_2\geq \frac{\sqrt{nq}}{5}\left(M_3 q^{-c_2} -\frac{\lambda q|J|}{2nq M_2}\right)\right).
\end{align*}
Let us bound $\|A\|_{2}$.
Observe that
\begin{align}\label{eq:calculnorme2A}
& \left\|\left[\left\{\left(\widehat{\Sigma}^{-1}-\Sigma^{-1}\right)\otimes X'\right\}\left\{(\Sigma^{1/2})'\otimes\textrm{Id}\right\}\right]_{J,\bullet}\right\|_2\nonumber\\
&=\rho\left(\left[\left(\widehat{\Sigma}^{-1}-\Sigma^{-1}\right)\Sigma\left(\widehat{\Sigma}^{-1}-\Sigma^{-1}\right)\otimes (X'X)\right]_{J,J}\right)^{1/2}\nonumber\\
&\leq \rho\left(\left(\widehat{\Sigma}^{-1}-\Sigma^{-1}\right)\Sigma\left(\widehat{\Sigma}^{-1}-\Sigma^{-1}\right)\right)^{1/2}\lambdaMax\left(X'X\right)^{1/2}\nonumber\\
&\leq \rho\left(\widehat{\Sigma}^{-1}-\Sigma^{-1}\right)\lambdaMax(\Sigma)^{1/2}\lambdaMax\left(X'X\right)^{1/2},\nonumber\\
\end{align}
where the first inequality comes from Theorem 4.3.15 of \cite{Horn:Johnson:1985}. Hence, by \ref{th2(v)}, \ref{th2(viii)} and \ref{th2(ix)}
\begin{equation}\label{eq:norme2_A}
\|A\|_{2}=\frac{1}{\sqrt{nq}}\left\|\left[\left\{\left(\widehat{\Sigma}^{-1}-\Sigma^{-1}\right)\otimes X'\right\}\left\{(\Sigma^{1/2})'\otimes\textrm{Id}\right\}\right]_{J,\bullet}\right\|_2
=O_P(q^{-1/2}(nq)^{-1/2}).
\end{equation}
By (\ref{eq:lambda_J_sqrt_nq}), (\ref{eq:CJJ-1:bound}) and (\ref{eq:norme2_A}), it is enough to prove that
$$
\PP\left(\sum_{k=1}^{nq} Z_k^2\geq nq\; n\; q^{-2c_2}\right)\to 0,\textrm { as } n\to\infty.
$$
The result follows from the Markov inequality and the first condition of \ref{assum:lambda_nq}. 

Let us now study the third term in the r.h.s of (\ref{eq:An_tilde_bound}). Observe that
\begin{align}\label{eq:WJ}
&W_J =\frac{1}{\sqrt{nq}}\left[\left(\Sigma^{-1/2}\otimes X'\right)
\left((\Sigma^{-1/2})'\otimes\textrm{Id}_{\rset^n}\right)\textrm{Vec}(E)\right]_J\nonumber\\
&\stackrel{d}{=}\frac{1}{\sqrt{nq}}\left[\left(\Sigma^{-1}\otimes X'\right)\left((\Sigma^{1/2})'\otimes\textrm{Id}_{\rset^n}\right)\right]_{J,\bullet}Z=:A_1 Z,
\end{align}
where $Z$ is a centered Gaussian random vector having a covariance matrix equal to identity and
\begin{equation}\label{eq:def_A1}
A_1=\frac{1}{\sqrt{nq}}\left[\left(\Sigma^{-1}\otimes X'\right)\left((\Sigma^{1/2})'\otimes\textrm{Id}_{\rset^n}\right)\right]_{J,\bullet}.
\end{equation}
Using (\ref{eq:Cauchy-Schwarz}), we get for all $j$ in $J$, for all $\gamma$ in $\rset$ and all $|J|\times |J|$ matrix $D$,
\begin{equation}\label{eq:D_WJ}
\PP\left(\left|\left(D\;W_J\right)_j\right|\geq \gamma\right)
=\PP\left(\left|\left(DA_1Z\right)_j\right|\geq \gamma\right)
\leq\PP\left(\|D\|_2\;\|A_1\|_2\; \|Z\|_2\geq \gamma\right),
\end{equation}
where $A_1$ is defined in (\ref{eq:def_A1}) and $Z$ is a centered Gaussian random vector having a covariance matrix equal to identity. Hence, for all $j$ in $J$,
\begin{align*}
&\PP\left(\left|\left(\left((\widetilde{C}_{J,J})^{-1}-(C_{J,J})^{-1}\right)W_J\right)_j\right|
\geq\frac{\sqrt{nq}}{5}\left(|\mathcal{B}_j|-\frac{\lambda}{2nq}\left|\left((C_{J,J})^{-1}\textrm{sign}(\mathcal{B}_J)\right)_j\right|\right)\right)\\
&\leq \PP\left(\left\|(\widetilde{C}_{J,J})^{-1}-(C_{J,J})^{-1}\right\|_2\;\|A_1\|_2\; \|Z\|_2\geq\frac{\sqrt{nq}}{5}\left(M_3 q^{-c_2} -\frac{\lambda q|J|}{2nq M_2}\right)\right).
\end{align*}
Let us now bound $\|A_1\|_2$. Note that
\begin{multline*}
 \left\|\left[\left(\Sigma^{-1}\otimes X'\right)\left((\Sigma^{1/2})'\otimes\textrm{Id}_{\rset^n}\right)\right]_{J,\bullet}\right\|_2
=\left\|\left[\left(\Sigma^{-1/2}\otimes X'\right)\right]_{J,\bullet}\right\|_2\\
=\rho\left(\left[\Sigma^{-1}\otimes (X'X)\right]_{J,J}\right)^{1/2}\leq \rho\left(\left[\Sigma^{-1}\otimes (X'X)\right]\right)^{1/2} 
\leq\lambdaMax(\Sigma^{-1})^{1/2}\lambdaMax(X'X)^{1/2},
\end{multline*}
where the first inequality comes from Theorem 4.3.15 of \cite{Horn:Johnson:1985}. Hence, by \ref{th2(v)} and \ref{th2(vii)},
\begin{equation}\label{eq:norme2_A1}
\|A_1\|_2\leq \frac{1}{nq}\lambdaMax(\Sigma^{-1})^{1/2}\lambdaMax(X'X)^{1/2}=O_P(q^{-1/2}).
\end{equation}
By (\ref{eq:lambda_J_sqrt_nq}), (\ref{eq:norm2_Ctilde_JJ-1_-C_JJ-1}) and  (\ref{eq:norme2_A1}) it is thus enough to prove that
$$
\PP\left(\sum_{k=1}^{nq}Z_k^2\geq nq\; n\; q^{-2c_2}\right)\to 0,\textrm{ as } n\to\infty.
$$
The result follows from the Markov inequality and the first condition of \ref{assum:lambda_nq}.

Let us now study the fourth term in the r.h.s of (\ref{eq:An_tilde_bound}). 
By (\ref{eq:B_WJ-WJtilde}), for all $j$ in $J$,
\begin{align*}
&\PP\left(\left|\left(\left((\widetilde{C}_{J,J})^{-1}-C_{J,J})^{-1}\right)\left(\widetilde{W}_J-W_J\right)\right)_j\right|
\geq\frac{\sqrt{nq}}{5}\left(|\mathcal{B}_j|-\frac{\lambda}{2nq}\left|\left((C_{J,J})^{-1}\textrm{sign}(\mathcal{B}_J)\right)_j\right|\right)\right)\\
&\leq \PP\left(\left\|(\widetilde{C}_{J,J})^{-1}-C_{J,J})^{-1}\right\|_2\; \|A\|_2\; \|Z\|_2
\geq \frac{\sqrt{nq}}{5}\left(M_3 q^{-c_2} -\frac{\lambda q|J|}{2nq M_2}\right)\right),
\end{align*}
where $A$ is defined in (\ref{eq:def_A}).

By (\ref{eq:lambda_J_sqrt_nq}), (\ref{eq:norm2_Ctilde_JJ-1_-C_JJ-1}) and (\ref{eq:norme2_A}), it is thus enough to prove that
$$
\PP\left(\sum_{k=1}^{nq}Z_k^2\geq (nq)\; n^2\; q^{1-2c_2}\right)\to 0,\textrm{ as } n\to\infty.
$$
The result follows from the Markov inequality and the fact that $c_2<1/2$.

Let us now study $\PP(\widetilde{B}_n)$. By definition of $\widetilde{B}_n$, we get that
$$
\PP(\widetilde{B}_n^c)=\PP\left(\left\{\left|\widetilde{C}_{J^c,J}(\widetilde{C}_{J,J})^{-1}\widetilde{W}_J-\widetilde{W}_{J^c}\right|\geq\frac{\lambda}{2\sqrt{nq}}\left(1-|\widetilde{C}_{J^c,J}(\widetilde{C}_{J,J})^{-1}\textrm{sign}(\mathcal{B}_J)|\right)\right\}\right).
$$
Observe that
\begin{align*}
\widetilde{C}_{J^c,J}(\widetilde{C}_{J,J})^{-1}\widetilde{W}_J-\widetilde{W}_{J^c}&=C_{J^c,J}(C_{J,J})^{-1}W_J-W_{J^c}\\
&+C_{J^c,J}(C_{J,J})^{-1}\left(\widetilde{W}_J-W_J\right)\\
&+C_{J^c,J}\left((\widetilde{C}_{J,J})^{-1}-(C_{J,J})^{-1}\right)W_J\\
&+C_{J^c,J}\left((\widetilde{C}_{J,J})^{-1}-(C_{J,J})^{-1}\right)\left(\widetilde{W}_J-W_J\right)\\
&+\left(\widetilde{C}_{J^c,J}- C_{J^c,J}\right)(C_{J,J})^{-1}W_J\\
&+\left(\widetilde{C}_{J^c,J}- C_{J^c,J}\right)(C_{J,J})^{-1}\left(\widetilde{W}_J-W_J\right)\\
&+\left(\widetilde{C}_{J^c,J}- C_{J^c,J}\right)\left((\widetilde{C}_{J,J})^{-1}-(C_{J,J})^{-1}\right)W_J\\
&+\left(\widetilde{C}_{J^c,J}- C_{J^c,J}\right)\left((\widetilde{C}_{J,J})^{-1}-(C_{J,J})^{-1}\right)\left(\widetilde{W}_J-W_J\right)\\
&+W_{J^c}-\widetilde{W}_{J^c}.
\end{align*}
Moreover,
\begin{align*}
\widetilde{C}_{J^c,J}(\widetilde{C}_{J,J})^{-1}\textrm{sign}(\mathcal{B}_J)&=C_{J^c,J}(C_{J,J})^{-1}\textrm{sign}(\mathcal{B}_J)\\
&+C_{J^c,J}\left((\widetilde{C}_{J,J})^{-1}-(C_{J,J})^{-1}\right)\textrm{sign}(\mathcal{B}_J)\\
&+\left(\widetilde{C}_{J^c,J}-C_{J^c,J}\right)(C_{J,J})^{-1}\textrm{sign}(\mathcal{B}_J)\\
&+\left(\widetilde{C}_{J^c,J}-C_{J^c,J}\right)\left((\widetilde{C}_{J,J})^{-1}-(C_{J,J})^{-1}\right)\textrm{sign}(\mathcal{B}_J).
\end{align*}
By \ref{eq:irrep} and the triangle inequality, we obtain that
\begin{align}\label{eq:Bn_tilde_bound}
&\PP(\widetilde{B}_n^c)\leq\PP\left(\left|C_{J^c,J}(C_{J,J})^{-1}W_J-W_{J^c}\right|\geq\frac{\lambda}{24\sqrt{nq}}\eta\right)\nonumber\\
&+\PP\left(\left|C_{J^c,J}(C_{J,J})^{-1}\left(\widetilde{W}_J-W_J\right)\right|\geq\frac{\lambda}{24\sqrt{nq}}\eta\right)\nonumber\\
 &+\PP\left(\left\{\left|C_{J^c,J}\left((\widetilde{C}_{J,J})^{-1}-(C_{J,J})^{-1}\right)W_J\right|\geq\frac{\lambda}{24\sqrt{nq}}\eta\right\}\right)\nonumber\\
 &+\PP\left(\left\{\left|C_{J^c,J}\left((\widetilde{C}_{J,J})^{-1}-(C_{J,J})^{-1}\right)\left(\widetilde{W}_J-W_J\right)\right|\geq\frac{\lambda}{24\sqrt{nq}}\eta\right\}\right)\nonumber\\
 &+\PP\left(\left\{\left|\left(\widetilde{C}_{J^c, J}- C_{J^c,J}\right)(C_{J,J})^{-1}W_J\right|\geq\frac{\lambda}{24\sqrt{nq}}\eta\right\}\right)\nonumber\\
 &+\PP\left(\left\{\left|\left(\widetilde{C}_{J^c ,J}- C_{J^c,J}\right)(C_{J,J})^{-1}\left(\widetilde{W}_J-W_J\right)\right|\geq\frac{\lambda}{24\sqrt{nq}}\eta\right\}\right)\nonumber\\
 &+\PP\left(\left\{\left|\left(\widetilde{C}_{J^c, J}- C_{J^c,J}\right)\left((\widetilde{C}_{J,J})^{-1}-(C_{J,J})^{-1}\right)W_J\right|\geq\frac{\lambda}{24\sqrt{nq}}\eta\right\}\right)\nonumber\\
 &+\PP\left(\left\{\left|\left(\widetilde{C}_{J^c, J}- C_{J^c,J}\right)\left((\widetilde{C}_{J,J})^{-1}-(C_{J,J})^{-1}\right)\left(\widetilde{W}_J-W_J\right)\right|\geq\frac{\lambda}{24\sqrt{nq}}\eta\right\}\right)\nonumber\\
 &+\PP\left(\left\{\left|W_{J^c}-\widetilde{W}_{J^c}\right|\geq\frac{\lambda}{24\sqrt{nq}}\eta\right\}\right)\nonumber\\
 &+\PP\left(\left\{\left|C_{J^c,J}\left((\widetilde{C}_{J,J})^{-1}-(C_{J,J})^{-1}\right)\textrm{sign}(\mathcal{B}_J)\right|\geq\frac{\eta}{12}\right\}\right)\nonumber\\
 &+\PP\left(\left\{\left|\left(\widetilde{C}_{J^c,J}-C_{J^c,J}\right)(C_{J,J})^{-1}\textrm{sign}(\mathcal{B}_J)\right|\geq\frac{\eta}{12}\right\}\right)\nonumber\\
 &+\PP\left(\left\{\left|\left(\widetilde{C}_{J^c,J}-C_{J^c,J}\right)\left((\widetilde{C}_{J,J})^{-1}-(C_{J,J})^{-1}\right)\textrm{sign}(\mathcal{B}_J)\right|\geq\frac{\eta}{12}\right\}\right).\nonumber\\
\end{align}

The first term in the r.h.s of (\ref{eq:Bn_tilde_bound}) tends to 0 by (\ref{eq:Bnc:limit}).

Let us now study the second term of (\ref{eq:Bn_tilde_bound}). 
By (\ref{eq:B_WJ-WJtilde}), we get that for all $j$ in $J^c$,
\begin{align*}
&\PP\left(\left|\left(C_{J^c,J}(C_{J,J})^{-1}\left(\widetilde{W}_J-W_J\right)\right)_j\right|
\geq\frac{\lambda}{24\sqrt{nq}}\eta\right)\\
&\leq \PP\left(\left\|C_{J^c,J}\right\|_2\|(C_{J,J})^{-1}\|_2\|A\|_2\|Z\|_2\geq \frac{\lambda}{24\sqrt{nq}}\eta\right).
\end{align*}
Observe that
\begin{align}\label{eq:n2_Cj_jc}
\left\|C_{J^c,J}\right\|_2 &= \rho\left(\frac{(\mathcal{X}_{\bullet,J^c})'\mathcal{X}_{\bullet,J}}{nq}\frac{(\mathcal{X}_{\bullet,J})'\mathcal{X}_{\bullet,J^c}}{nq}\right)^{1/2}
=\frac{1}{nq}\left\|(\mathcal{X}_{\bullet,J^c})'\mathcal{X}_{\bullet,J}\right\|_2\leq \frac{\left\|(\mathcal{X}_{\bullet,J^c})'\right\|_2}{\sqrt{nq}}\frac{\left\|\mathcal{X}_{\bullet,J}\right\|_2}{\sqrt{nq}}\nonumber\\
&\leq\rho\left(\frac{(\mathcal{X}_{\bullet,J^c})'\mathcal{X}_{\bullet,J^c}}{nq}\right)^{1/2}
\rho\left(\frac{(\mathcal{X}_{\bullet,J})'\mathcal{X}_{\bullet,J}}{nq}\right)^{1/2} \nonumber \\
&=\rho(C_{J^cJ^c})^{1/2}\rho(C_{J,J})^{1/2}\leq\rho(C)=\frac{\lambdaMax(\Sigma^{-1})}{q}\lambdaMax(X'X/n)=O_P(q^{-1}).
\end{align}
In (\ref{eq:n2_Cj_jc}) the last inequality and  the fourth equality come from Theorem 4.3.15 of \cite{Horn:Johnson:1985} and (\ref{eq:C_C_tilde}), respectively. The last equality 
comes from \ref{th2(v)} and \ref{th2(vii)}.
%

By (\ref{eq:CJJ-1:bound}), (\ref{eq:norme2_A}) and (\ref{eq:n2_Cj_jc}), it is thus enough to prove that
$$
\PP\left(\sum_{k=1}^{nq} Z_k^2\geq \left((nq)^{1/2} \sqrt{q}\frac{\lambda}{\sqrt{nq}}\right)^2\right)=\PP\left(\sum_{k=1}^{nq} Z_k^2\geq (nq)
\left(\frac{\lambda}{\sqrt{n}}\right)^2\right)\to 0,\textrm{ as } n\to\infty,
$$
which holds true by the second condition of \ref{assum:lambda_nq} and Markov inequality.  Hence, the second term of (\ref{eq:Bn_tilde_bound}) tends to zero as $n$ tends to infinity.

Let us now study the third term of (\ref{eq:Bn_tilde_bound}). 
By (\ref{eq:D_WJ}), we get that for all $j$ in $J^c$,
\begin{align*}
&\PP\left(\left|\left(C_{J^c,J}\left((\widetilde{C}_{J,J})^{-1}-(C_{J,J})^{-1}\right)W_J\right)_j\right|\geq\frac{\lambda}{24\sqrt{nq}}\eta\right)\\
&\leq\PP\left(\left\|C_{J^c,J}\right\|_2\|(\widetilde{C}_{J,J})^{-1}-(C_{J,J})^{-1}\|_2\|A_1\|_2\|Z\|_2\geq \frac{\lambda}{24\sqrt{nq}}\eta\right).
\end{align*}
By (\ref{eq:norm2_Ctilde_JJ-1_-C_JJ-1}), (\ref{eq:norme2_A1}) and (\ref{eq:n2_Cj_jc}), it is thus enough to prove that
$$
\PP\left(\sum_{k=1}^{nq} Z_k^2\geq \left((nq)^{1/2} \sqrt{q}\frac{\lambda}{\sqrt{nq}}\right)^2\right)=\PP\left(\sum_{k=1}^{nq} Z_k^2\geq (nq)
\left(\frac{\lambda}{\sqrt{n}}\right)^2\right)\to 0,\textrm{ as } n\to\infty,
$$
which holds true by  the second condition of \ref{assum:lambda_nq} and Markov inequality. Hence, the third term of (\ref{eq:Bn_tilde_bound}) tends to zero as $n$ tends to infinity.

Let us now study the fourth term of (\ref{eq:Bn_tilde_bound}). By (\ref{eq:B_WJ-WJtilde}), it amounts to prove that
$$
\PP\left(\left\|C_{J^c,J}\right\|_2\|(\widetilde{C}_{J,J})^{-1}-(C_{J,J})^{-1}\|_2\|A\|_2\|Z\|_2\geq \frac{\lambda}{24\sqrt{nq}}\eta\right)\to 0,\textrm{ as } n\to\infty.
$$
By (\ref{eq:n2_Cj_jc}), (\ref{eq:norm2_Ctilde_JJ-1_-C_JJ-1}) and (\ref{eq:norme2_A}) it is enough tho prove that
$$
\PP\left(\sum_{k=1}^{nq} Z_k^2\geq(nq) \;(nq)\left(\frac{\lambda}{\sqrt{n}}\right)^2\right)\to 0,\textrm{ as } n\to\infty,
$$
which holds true by the second condition of \ref{assum:lambda_nq}. Hence, the fourth term of (\ref{eq:Bn_tilde_bound}) tends to zero as $n$ tends to infinity.

Let us now study the fifth term of (\ref{eq:Bn_tilde_bound}). 
By (\ref{eq:D_WJ}), proving that the fifth term of (\ref{eq:Bn_tilde_bound}) tends to 0 amounts to proving that
$$
\PP\left(\left\|C_{J^c,J}-\widetilde{C}_{J^c, J}\right\|_2\|(C_{J,J})^{-1}\|_2\|A_1\|_2\|Z\|_2\geq \frac{\lambda}{24\sqrt{nq}}\eta\right)\to 0,\textrm{ as } n\to\infty.
$$
Let us now bound $\|C_{J^c,J}-\widetilde{C}_{J^c, J}\|_2$.
\begin{align}\label{eq:norme2_CJcJ-CJcJ_tilde}
&\left\|C_{J^c,J}-\widetilde{C}_{J^c ,J}\right\|_2=\left\|\left(C-\widetilde{C}\right)_{J^c, J}\right\|_2
=\rho\left(\left(C-\widetilde{C}\right)_{J^c, J}\left(C-\widetilde{C}\right)_{J^c, J}\right)^{1/2}\nonumber\\
&\leq\left\|\left(C-\widetilde{C}\right)_{J^c, J}\left(C-\widetilde{C}\right)_{J^c, J}\right\|_{\infty}^{1/2}
\leq \left\|\left(C-\widetilde{C}\right)\left(C-\widetilde{C}\right)\right\|_{\infty}^{1/2}
\leq\left\|C-\widetilde{C}\right\|_{\infty}\nonumber\\
&=\frac{1}{q}\left\|\Sigma^{-1}-\widehat{\Sigma}^{-1}\right\|_{\infty}\;\left\|\frac{X'X}{n}\right\|_{\infty}=O_P(q^{-1}(nq)^{-1/2}),
\end{align}
as $n$ tends to infinity, where the last equality comes from \ref{th2(v)} and \ref{th2(ix)}.

 By (\ref{eq:CJJ-1:bound}), (\ref{eq:norme2_A1}) and (\ref{eq:norme2_CJcJ-CJcJ_tilde}), to prove that the fifth term of (\ref{eq:Bn_tilde_bound}) tends to zero as $n$ tends to infinity, it is enough
to prove that
$$
\PP\left(\sum_{k=1}^{nq} Z_k^2\geq nq\left(\frac{\lambda}{\sqrt{n}}\right)^2\right)\to 0,\textrm{ as } n\to\infty,
$$
which holds using Markov's inequality and the second condition of \ref{assum:lambda_nq}.

 
Using similar arguments as those used for proving that the second, third and fourth terms of (\ref{eq:Bn_tilde_bound}) tend to zero, 
we get that the sixth, seventh and eighth terms of (\ref{eq:Bn_tilde_bound}) tend to zero, as $n$ tends to infinity, by replacing (\ref{eq:n2_Cj_jc}) by (\ref{eq:norme2_CJcJ-CJcJ_tilde}).


 Let us now study the ninth term of (\ref{eq:Bn_tilde_bound}). 
Replacing $J$ by $J^c$ in (\ref{eq:diff_Wn}), (\ref{eq:def_A}), (\ref{eq:B_WJ-WJtilde}, (\ref{eq:calculnorme2A}) and (\ref{eq:norme2_A}) in order to prove that  the ninth term of 
(\ref{eq:Bn_tilde_bound}) tends to 0 it is enough to prove that 
$$
\PP\left(\sum_{k=1}^{nq} Z_k^2\geq nq\left(\frac{\lambda}{\sqrt{n}}\right)^2\right)\to 0,\textrm{ as } n\to\infty,
$$
which holds using Markov's inequality and the second condition of \ref{assum:lambda_nq}.

Let us now study the tenth term of (\ref{eq:Bn_tilde_bound}). Using the same idea as the one used for proving (\ref{eq:U_s_j}), we get that
\begin{align*}
&\PP\left(\left\{\left|C_{J^c,J}\left((\widetilde{C}_{J,J})^{-1}-(C_{J,J})^{-1}\right)\textrm{sign}(\mathcal{B}_J)\right|\geq\frac{\eta}{12}\right\}\right)\\
&\leq\PP\left(\sqrt{|J|}\left\|C_{J^c,J}\right\|_2\; \left\|(\widetilde{C}_{J,J})^{-1}-(C_{J,J})^{-1}\right\|_2\geq\frac{\eta}{12}\right),
\end{align*}
which tends to zero as $n$ tends to infinity by \ref{th1(iii)},  (\ref{eq:norm2_Ctilde_JJ-1_-C_JJ-1}), (\ref{eq:n2_Cj_jc}) and the fact that $c_1<1/2$.

Let us now study the eleventh term of (\ref{eq:Bn_tilde_bound}). Using the same idea as the one used for proving (\ref{eq:U_s_j}), we get that
\begin{align*}
&\PP\left(\left\{\left|\left(\widetilde{C}_{J^c,J}-C_{J^c,J}\right)(C_{J,J})^{-1}\textrm{sign}(\mathcal{B}_J)\right|\geq\frac{\eta}{12}\right\}\right)\\
&\leq \PP\left(\sqrt{|J|}\left\|\widetilde{C}_{J^c,J}-C_{J^c,J}\right\|_2 \|(C_{J,J})^{-1}\|_2\geq\frac{\eta}{12}\right),
\end{align*}
which tends to zero as $n$ tends to infinity by \ref{th1(iii)}, (\ref{eq:CJJ-1:bound}) and (\ref{eq:norme2_CJcJ-CJcJ_tilde}) and the fact that $c_1<1/2$.

Finally, the twelfth term of (\ref{eq:Bn_tilde_bound}) can be bounded as follows:
\begin{align*}
&\PP\left(\left\{\left|\left(\widetilde{C}_{J^c,J}-C_{J^c,J}\right)\left((\widetilde{C}_{J,J})^{-1}-(C_{J,J})^{-1}\right)\textrm{sign}(\mathcal{B}_J)\right|\geq\frac{\eta}{12}\right\}\right)\\
&\leq \PP\left(\sqrt{|J|} \left\|\widetilde{C}_{J^c,J}-C_{J^c,J}\right\|_2\left\|(\widetilde{C}_{J,J})^{-1}-(C_{J,J})^{-1}\right\|_2\geq\frac{\eta}{12}\right),
\end{align*}
which tends to zero as $n$ tends to infinity by \ref{th1(iii)}, (\ref{eq:norm2_Ctilde_JJ-1_-C_JJ-1}) and (\ref{eq:norme2_CJcJ-CJcJ_tilde}) and the fact that $c_1<1/2$.
\end{proof}

\begin{proof}[Proof of Proposition \ref{prop_IC}]
Observe that
\begin{equation}\label{eq:Sigma-1}
\Sigma^{-1} =\left(
\begin{matrix}
1 & -\phi_1 & 0 & \cdots & 0\\
-\phi_1 & 1+\phi_1^2 & -\phi_1 & \cdots & 0 \\
0 &-\phi_1 & \ddots & \ddots & \vdots \\
\vdots & \vdots &  \ddots & 1+\phi_1^2 & -\phi_1 \\
0 & 0 & \cdots  &  -\phi_1  &  1 &
\end{matrix}
\right).
\end{equation}
Let 
$
S=\mathcal{X'}\mathcal{X}=\Sigma^{-1} \otimes X'X.
$
Then,
\begin{equation*}
S_{i,j} =\left\{
\begin{tabular}{cl}
$n_{r_i+1}$ & if  $j=i$ and  $k_i\in\{0,q-1\}$\\
$(1+\phi_1^2)n_{r_i+1}$ & if  $j=i$ and  $k_i\notin\{ 0,q-1\}$ \\
$-\phi_1 n_{r_i+1}$ & if $j=i+p$ or if  $j=i-p$ \\
0 & \textrm{otherwise} 
\end{tabular}
\right.,
\end{equation*}
where $i-1=(p-1)k_i+r_i$ corresponds to the Euclidean division of $(i-1)$ by $(p-1)$.

In order to prove \ref{eq:irrep}, it is enough to prove that 
$$
\|S_{J^c,J} (S_{J,J})^{-1}\|_{\infty}\leq 1-\eta,
$$
where $\eta\in (0,1)$.

Since for all $j$, $(j-p)\in J^c$ or $(j+p)\in J^c$,
$$
\|S_{J^c,J}\|_\infty=\nu |\phi_1|.
$$
Let $A=S_{J,J}$. Since $A=(a_{i,j})$ is a diagonally dominant matrix, then, by Theorem 1 of \cite{varah:1975},
$$
\|A^{-1}\|_{\infty}\leq\frac{1}{\min_k (a_{k,k}-\displaystyle{\sum_{\stackrel{1\leq j\leq |J|}{j \neq k}}a_{k,j}})}.
$$
Using that for all $j$, $(j-p)\in J^c$ or $(j+p)\in J^c$,
$$
\sum_{\stackrel{1\leq j\leq |J|}{j \neq k}} a_{k,j}\leq \nu |\phi_1|.
$$
If $k\in J$ then $k>p$ and $k<pq -p$. Thus,
$$
a_{k,k}\geq \nu(1+\phi_1^2).
$$
Hence,
$$
\|A^{-1}\|_{\infty}\leq\frac{1}{\nu(1 + \phi_1^2- |\phi_1|)}
$$
and
$$
\|S_{J^c,J}(S_{J,J})^{-1}\|_\infty \leq \|S_{J^c,J}\|_\infty \|(S_{J,J})^{-1}\|_\infty \leq \frac{ |\phi_1| }{1 + \phi_1^2-|\phi_1|}.
$$
Since $|\phi_1|<1$, the strong Irrepresentability Condition holds when 
$$
|\phi_1| \leq (1-\eta)(1 +|\phi_1|^2 - |\phi_1|),
$$
which is true for a small enough $\eta$.
\end{proof}

\begin{proof}[Proof of Proposition \ref{prop:cond_th_AR}]
Since $|\phi_1|<1$,
$$
\|\Sigma^{-1}\|_\infty\leq |\phi_1|+|1+\phi_1^2|\leq 3,
$$
which gives \ref{th2(vii)} by Theorem 5.6.9 of \cite{Horn:Johnson:1985}.

Observe that
$$
\|\Sigma\|_{\infty}\leq \frac{1}{1-\phi_1^2}\left(1+2\sum_{h=1}^{q-1}|\phi_1|^h\right)\leq\frac{1}{1-\phi_1^2}\left(1+\frac{2}{1-|\phi_1|}\right)=\frac{3-|\phi_1|}{1-\phi_1^2}
\leq \frac{3}{1-\phi_1^2},
$$
which gives \ref{th2(viii)} by Theorem 5.6.9 of \cite{Horn:Johnson:1985}.

Since $\widehat{\Sigma}^{-1}$ has the same expression as $\Sigma^{-1}$ defined in (\ref{eq:Sigma-1}) except that $\phi_1$ is replaced by $\widehat{\phi_1}$ defined in 
(\ref{eq:phi_1_hat}), we get that
$$
\left\|\Sigma^{-1}-\widehat{\Sigma}^{-1}\right\|_{\infty}\leq 2\left|\phi_1-\widehat{\phi}_1\right|+\left(\phi_1-\widehat{\phi}_1\right)^2,
$$
which implies Assumption \ref{th2(ix)} of Theorem \ref{th2} by Lemma \ref{lem:AR:estim}.


Let us now check Assumption \ref{th2(x)} of Theorem \ref{th2}. Since, by Theorem 5.6.9 of \cite{Horn:Johnson:1985},
$\rho(\Sigma-\widehat{\Sigma})\leq \|\Sigma-\widehat{\Sigma}\|_{\infty}$, it is enough to prove that
$$
\left\|\Sigma-\widehat{\Sigma}\right\|_{\infty}=O_P((nq)^{-1/2}),\textrm{ as } n\to\infty.
$$
Observe that
\begin{align*}
&\left\|\Sigma-\widehat{\Sigma}\right\|_{\infty}\leq \left|\frac{1}{1-\phi_1^2}-\frac{1}{1-\widehat{\phi}_1^2}\right|
+2\sum_{h=1}^{q-1}\left|\frac{\phi_1^h}{1-\phi_1^2}-\frac{\widehat{\phi}_1^h}{1-\widehat{\phi}_1^2}\right|\\
&\leq \left|\frac{\phi_1^2-\widehat{\phi}_1^2}{(1-\phi_1^2)(1-\widehat{\phi}_1^2)}\right|+2\sum_{h=1}^{q-1}\left|\frac{\phi_1^h-\widehat{\phi}_1^h}{1-\phi_1^2}\right|
+2\sum_{h=1}^{q-1}\left|\widehat{\phi}_1^h\left(\frac{1}{1-\phi_1^2}-\frac{1}{1-\widehat{\phi}_1^2}\right)\right|\\
&\leq \left|\frac{(\phi_1-\widehat{\phi}_1)(\phi_1+\widehat{\phi}_1)}{(1-\phi_1^2)(1-\widehat{\phi}_1^2)}\right|+2\sum_{h=1}^{q-1}\left|\frac{\phi_1^h-\widehat{\phi}_1^h}{1-\phi_1^2}\right|
+2\sum_{h=1}^{q-1}\left|\left(\widehat{\phi}_1^h-\phi_1^h\right)\left(\frac{1}{1-\phi_1^2}-\frac{1}{1-\widehat{\phi}_1^2}\right)\right|\\
&\hspace{50mm}+2\sum_{h=1}^{q-1}\left|\phi_1^h \left(\frac{1}{1-\phi_1^2}-\frac{1}{1-\widehat{\phi}_1^2}\right)\right|\\
&\leq \left|\frac{(\phi_1-\widehat{\phi}_1)(\phi_1+\widehat{\phi}_1)}{(1-\phi_1^2)(1-\widehat{\phi}_1^2)}\right|\left(1+\frac{2}{1-|\phi_1|}\right)
+2\left(\frac{1}{|1-\phi_1^2|}+\left|\frac{(\phi_1-\widehat{\phi}_1)(\phi_1+\widehat{\phi}_1)}{(1-\phi_1^2)(1-\widehat{\phi}_1^2)}\right|\right)
\sum_{h=1}^{q-1}\left|\widehat{\phi}_1^h-\phi_1^h\right|.
\end{align*}
Moreover,
\begin{align*}
&\sum_{h=1}^{q-1}\left|\widehat{\phi}_1^h-\phi_1^h\right|\leq \left|\widehat{\phi}_1-\phi_1\right|\sum_{h=1}^{q-1}\sum_{k=0}^{h-1}
|\phi_1|^k |\widehat{\phi}_1|^{h-k-1} \leq \left|\widehat{\phi}_1-\phi_1\right|\left(\frac{1-|\widehat{\phi}_1|^{q-1}}{1-|\widehat{\phi}_1|}\right)
\left(\frac{1-|\phi_1|^{q-1}}{1-|\phi_1|}\right)\\
&\leq \left|\widehat{\phi}_1-\phi_1\right|\left(\frac{1}{1-|\widehat{\phi}_1|}\right)
\left(\frac{1}{1-|\phi_1|}\right).
\end{align*}
Thus,  by Lemma \ref{lem:AR:estim},
$$
\left\|\Sigma-\widehat{\Sigma}\right\|_{\infty}=O_P((nq)^{-1/2}),
$$
which implies Assumption \ref{th2(x)} of Theorem \ref{th2}.
\end{proof}

\begin{proof}[Proof of Lemma \ref{lem:AR:estim}]
In the following, for notational simplicity, $q=q_n$.
Observe that
$$
\sqrt{nq}\widehat{\phi}_1=\frac{\frac{1}{\sqrt{nq}}\sum_{i=1}^n\sum_{\ell=2}^q \widehat{E}_{i,\ell} \widehat{E}_{i,\ell-1}}{\frac{1}{nq}\sum_{i=1}^n\sum_{\ell=1}^{q-1} \widehat{E}_{i,\ell}^2}.
$$
By (\ref{eq: Ehat}),
\begin{align}
\sum_{i=1}^n\sum_{\ell=2}^q \widehat{E}_{i,\ell} \widehat{E}_{i,\ell-1}&=\sum_{\ell=2}^q (\widehat{E}_{\bullet,\ell})'\widehat{E}_{\bullet,\ell-1}
=\sum_{\ell=2}^q (\Pi E_{\bullet,\ell})' (\Pi E_{\bullet,\ell-1})\nonumber\\
&=\sum_{\ell=2}^q (\phi_1 \Pi E_{\bullet,\ell-1} +\Pi Z_{\bullet,\ell})' (\Pi E_{\bullet,\ell-1})\label{eq:phi1_hat_def_E}\\
&=\phi_1\sum_{\ell=1}^{q-1}(\Pi E_{\bullet,\ell})'(\Pi E_{\bullet,\ell})+\sum_{\ell=2}^q (\Pi Z_{\bullet,\ell})'(\Pi E_{\bullet,\ell-1})\nonumber,
\end{align}
where (\ref{eq:phi1_hat_def_E}) comes from the definition of $(E_{i,t})$.

Hence,
$$
\sqrt{nq}(\widehat{\phi}_1-\phi_1)=\frac{\frac{1}{\sqrt{nq}}\sum_{\ell=2}^q (\Pi Z_{\bullet,\ell})'(\Pi E_{\bullet,\ell-1})}{\frac{1}{nq}\sum_{i=1}^n\sum_{\ell=1}^{q-1} \widehat{E}_{i,\ell}^2}.
$$
In order to prove that $\sqrt{nq}(\widehat{\phi}_1-\phi_1)=O_P(1)$, it is enough to prove that
\begin{equation}\label{eq:den:approx}
\frac{1}{nq}\sum_{i=1}^n\sum_{\ell=1}^{q-1} E_{i,\ell}^2-\frac{1}{nq}\sum_{i=1}^n\sum_{\ell=1}^{q-1} \widehat{E}_{i,\ell}^2=o_P(1),\textrm{ as }
n\to\infty,
\end{equation}
by Lemma \ref{lem:denom} and
\begin{equation}\label{eq:num:approx}
\frac{1}{\sqrt{nq}}\sum_{\ell=2}^q (\Pi Z_{\bullet,\ell})'(\Pi E_{\bullet,\ell-1})=O_P(1),\textrm{ as } n\to\infty.
\end{equation}

Let us first prove (\ref{eq:den:approx}). By (\ref{eq:Ehat_1}), 
$$
\widehat{\mathcal{E}}=\left[\textrm{Id}_{\rset^q}\otimes\Pi\right]\mathcal{E}:= A\mathcal{E}.
$$
Note that
$$
\Cov(\widehat{\mathcal{E}})=A(\Sigma\otimes\textrm{Id}_{\rset^n})A'=\Sigma\otimes\Pi.
$$
Hence, for all $i$
$$
\Var(\widehat{\mathcal{E}}_i)\leq\lambdaMax(\Sigma).
$$
Since the covariance matrix of $\mathcal{E}$ is equal to $\Sigma\otimes\textrm{Id}_{\rset^n}$, for all $i$
$$
\Var(\mathcal{E}_i)\leq\lambdaMax(\Sigma).
$$
By Markov inequality,
\begin{multline*}
\frac{1}{nq}\sum_{i=1}^n\sum_{\ell=1}^{q-1} E_{i,\ell}^2-\frac{1}{nq}\sum_{i=1}^n\sum_{\ell=1}^{q-1} \widehat{E}_{i,\ell}^2
=\frac{1}{nq}\sum_{i=1}^n\sum_{\ell=1}^{q} E_{i,\ell}^2-\frac{1}{nq}\sum_{i=1}^n\sum_{\ell=1}^{q} \widehat{E}_{i,\ell}^2+o_P(1)\\
=\frac{1}{nq}\left(\|\mathcal{E}\|_2^2-\|\widehat{\mathcal{E}}\|_2^2\right)+o_P(1).
\end{multline*}
Observe that
\begin{align*}
&\|\mathcal{E}\|_2^2-\|\widehat{\mathcal{E}}\|_2^2=\|\mathcal{E}\|_2^2-\|A\mathcal{E}\|_2^2=\mathcal{E}'\mathcal{E}-\mathcal{E}'A'A\mathcal{E}
=\mathcal{E}'\left(\textrm{Id}_{\rset^{nq}}-\textrm{Id}_{\rset^q}\otimes\Pi\right)\mathcal{E}\\
&=\mathcal{E}'\left(\textrm{Id}_{\rset^q}\otimes\left(\textrm{Id}_{\rset^n}-\Pi\right)\right)\mathcal{E}=\sum_{i=1}^{pq}\widetilde{\mathcal{E}}_i^2,
\end{align*}
where $\widetilde{\mathcal{E}}=O\mathcal{E}$, where $O$ is an orthogonal matrix. 
Using that
$$
\PE(\widetilde{\mathcal{E}}_i^2)=\Cov(\widetilde{\mathcal{E}})_{i,i}\leq\lambdaMax(\Sigma),
$$
and Markov inequality, we get (\ref{eq:den:approx}).

Let us now prove (\ref{eq:num:approx}). 
By definition of $(E_{i,t})$ and since $|\phi_1|<1$, $\PE[(\Pi Z_{\bullet,\ell})'(\Pi E_{\bullet,\ell-1})]=0$. Moreover,
\begin{align*}
&\PE\left[\left(\sum_{\ell=2}^q (\Pi Z_{\bullet,\ell})'(\Pi E_{\bullet,\ell-1})\right)^2\right]
=\PE\left[\left(\sum_{\ell=2}^q\sum_{i=1}^n \left(\sum_{k=1}^n\Pi_{i,k} Z_{k,\ell}\right)\left(\sum_{j=1}^n\Pi_{i,j} E_{j,\ell-1}\right)\right)^2\right]\\
&=\sum_{2\leq\ell,\ell'\leq q}\;\sum_{1\leq i,j,k,i',j',k'\leq n} \Pi_{i,k}\Pi_{i',k'}\Pi_{i,j}\Pi_{i',j'}\PE\left(Z_{k,\ell}Z_{k',\ell'}E_{j,\ell-1}E_{j',\ell'-1}\right)\\
&=\sum_{2\leq\ell,\ell'\leq q}\;\sum_{1\leq i,j,k,i',j',k'\leq n} \Pi_{i,k}\Pi_{i',k'}\Pi_{i,j}\Pi_{i',j'}\sum_{r,s\geq 0}\phi_1^r \phi_1^s
\PE\left(Z_{k,\ell}Z_{k',\ell'}Z_{j,\ell-1-r}Z_{j',\ell'-1-s}\right),
\end{align*}
since the $(E_{i,t})$ are AR(1) processes with $|\phi_1|<1$. Note that $\PE(Z_{k,\ell}Z_{k',\ell'}Z_{j,\ell-1-r}Z_{j',\ell'-1-s})=0$ except when $\ell=\ell'$, $k=k'$, $j=j'$ and $r=s$.

Thus,
\begin{align*}
&\PE\left[\left(\sum_{\ell=2}^q (\Pi Z_{\bullet,\ell})'(\Pi E_{\bullet,\ell-1})\right)^2\right]
=\sigma^4\left(\sum_{r\geq 0}\phi_1^{2r}\right)\sum_{\ell=2}^q\sum_{1\leq i,j,k,i'\leq n}\Pi_{i,k}\Pi_{i',k}\Pi_{i,j}\Pi_{i',j}\\
&=\frac{q\sigma^4}{1-\phi_1^2}\textrm{Tr}(\Pi)\leq\frac{nq\sigma^4}{1-\phi_1^2},
\end{align*}
where $\textrm{Tr}(\Pi)$ denotes the trace of $\Pi$, which concludes the proof of (\ref{eq:num:approx}) by Markov inequality.

\end{proof}

\section{Technical lemmas}\label{sec:lemmas}



\begin{lemma}\label{lem:proj}
Let $A\in\mathcal{M}_n(\rset)$ and $\Pi$ an orthogonal projection matrix.
For any $j$ in $\{1,\ldots,n\}$  
$$
(A'\Pi A)_{jj}\geq 0.
$$ 
\end{lemma}

\begin{proof}[Proof of Lemma \ref{lem:proj}]
Observe that
$$
(A'\Pi A)=A'\Pi'\Pi A=(\Pi A)'(\Pi A),
$$
since $\Pi$ is an orthogonal projection matrix.
Moreover,
$$
(A'\Pi A)_{jj}=e_j'(\Pi A)'(\Pi A)e_j\geq 0,
$$
since $(\Pi A)'(\Pi A)$ is a positive semidefinite symmetric matrix, where $e_j$ is a vector containing null entries except the $j$th entry which is equal to 1.
\end{proof}

\begin{lemma}\label{lem:denom}
Assume that $(E_{1,t})_t$, $(E_{2,t})_t$, ..., $(E_{n,t})_t$ are independent AR(1) processes satisfying:
$$
E_{i,t}-\phi_1 E_{i,t-1}=Z_{i,t},\; \forall i\in\{1,\dots,n\},
$$
where the $Z_{i,t}$'s are zero-mean i.i.d. Gaussian random variables with variance $\sigma^2$ and $|\phi_1|<1$. 
Then,
$$
\frac{1}{nq_n}\sum_{i=1}^n\sum_{\ell=1}^{q_n-1} E_{i,\ell}^2\stackrel{P}{\longrightarrow}\frac{\sigma^2}{1-\phi_1^2},\textrm{ as } n\to\infty.
$$
\end{lemma}

\begin{proof}[Proof of Lemma \ref{lem:denom}]
In the following, for notational simplicity, $q=q_n$.
Since $\PE(E_{i,\ell}^2)=\sigma^2/(1-\phi_1^2)$, it is enough to prove that
$$
\frac{1}{nq}\sum_{i=1}^n\sum_{\ell=1}^{q-1} \left(E_{i,\ell}^2-\PE(E_{i,\ell}^2)\right)\stackrel{P}{\longrightarrow} 0,\textrm{ as } n\to\infty.
$$
Since 
$$
E_{i,\ell}^2=\left(\sum_{j\geq 0}\phi_1^j Z_{i,\ell-j}\right)^2=\sum_{j,j'\geq 0}\phi_1^j\phi_1^{j'} Z_{i,\ell-j} Z_{i,\ell-j'},
$$
\begin{align}\label{eq:variance}
&\Var\left(\frac{1}{nq}\sum_{i=1}^n\sum_{\ell=1}^{q-1} \left(E_{i,\ell}^2-\PE(E_{i,\ell}^2)\right)\right)
=\frac{1}{(nq)^2}\sum_{i=1}^n\;\sum_{1\leq\ell,\ell'\leq q-1} \Cov(E_{i,\ell}^2 ; E_{i,\ell'}^2)\nonumber\\
&=\frac{1}{(nq)^2}\sum_{i=1}^n\;\sum_{1\leq\ell,\ell'\leq q-1}\;\sum_{j,j'\geq 0}\;\sum_{k,k'\geq 0}\;\phi_1^j\phi_1^{j'}\phi_1^k\phi_1^{k'}\Cov(Z_{i,\ell-j} Z_{i,\ell-j'} ; Z_{i,\ell'-k} Z_{i,\ell'-k'}).
\end{align}
By Cauchy-Schwarz inequality $|\Cov(Z_{i,\ell-j} Z_{i,\ell-j'} ; Z_{i,\ell'-k} Z_{i,\ell'-k'})|$ is bounded by a positive constant. Moreover $\sum_{j\geq 0} |\phi_1|^j<\infty$, hence 
(\ref{eq:variance}) tends to zero as $n$ tend to infinity, which concludes the proof of the lemma.
\end{proof}

\bibliographystyle{chicago}
\bibliography{biblio}

\end{document}